\documentclass[letterpaper,11pt]{amsart}

\usepackage[margin=1.2in]{geometry}
\usepackage{amsmath,amsthm,amssymb}
\usepackage{xspace,xcolor}
\usepackage[
  breaklinks,colorlinks,
  citecolor=teal,linkcolor=teal,urlcolor=teal
]{hyperref}
\usepackage[all]{xy}

\theoremstyle{plain}
\newtheorem{thm}{Theorem}[section]
\newtheorem*{thmA}{Theorem A}
\newtheorem*{thmB}{Theorem B}
\newtheorem*{thmC}{Theorem C}

\newtheorem{lem}[thm]{Lemma}
\newtheorem{prop}[thm]{Proposition}
\newtheorem{cor}[thm]{Corollary}

\theoremstyle{definition}
\newtheorem{defi}[thm]{Definition}

\theoremstyle{remark}
\newtheorem{rmk}[thm]{Remark}
\newtheorem{eg}[thm]{Example}

\newtheorem{claim}[thm]{Claim}

\numberwithin{equation}{section}

\def\Mustata{Mus\-ta\-\c{t}\u{a}\xspace}

\def\Z{{\mathbf Z}}
\def\Q{{\mathbf Q}}
\def\R{{\mathbf R}}
\def\C{{\mathbf C}}

\def\QQ{\texorpdfstring{$\Q$}{Q}}

\def\A{{\mathbf A}}
\def\P{{\mathbf P}}

\def\cE{\mathcal{E}}
\def\cF{\mathcal{F}}

\def\cH{\mathcal{H}}

\def\I{\mathcal{I}}
\def\J{\mathcal{J}}
\def\O{\mathcal{O}}

\def\fra{\mathfrak{a}}
\def\frb{\mathfrak{b}}
\def\frc{\mathfrak{c}}
\def\frd{\mathfrak{d}}
\def\frj{\mathfrak{j}}
\def\frm{\mathfrak{m}}
\def\frn{\mathfrak{n}}

\def\frA{\mathfrak{A}}
\def\frB{\mathfrak{B}}

\def\a{\alpha}
\def\b{\beta}
\def\g{\gamma}

\def\f{\phi}
\def\ff{\psi}
\def\e{\eta}
\def\ep{\epsilon}

\def\l{\lambda}
\def\n{\nu}
\def\m{\mu}
\def\om{\omega}
\def\p{\pi}

\def\x{\xi}

\def\Om{\Omega}

\def\.{\cdot}
\def\^{\widehat}
\def\~{\widetilde}

\def\ov{\overline}

\def\inj{\hookrightarrow}

\def\lrd{\lfloor}
\def\rrd{\rfloor}
\def\({\left(}
\def\){\right)}

\newcommand{\rd}[1]{\lrd{#1}\rrd}

\renewcommand{\and}{ \ \ \text{ and } \ \ }
\renewcommand{\for}{ \ \ \text{ for } \ \ }
\newcommand{\fall}{ \ \ \text{ for all } \ \ }

\newcommand{\ku}[2]{k_{#1}({#2})} 
\newcommand{\au}[2]{a_{#1}({#2})} 
\newcommand{\Ku}[2]{K_{{#1}/{#2}}} 
\newcommand{\mld}[2]{\operatorname{mld}_{#1}({#2})}
\newcommand{\lct}[1]{\operatorname{lct}({#1})}
\newcommand{\MIu}[1]{\J\big({#1}\big)}

\newcommand{\kj}[2]{k^\diamond_{#1}({#2})} 
\newcommand{\aj}[2]{a^\diamond_{#1}({#2})} 
\newcommand{\Kj}[2]{K^\diamond_{{#1}/{#2}}} 
\newcommand{\jmld}[2]{\operatorname{mld}^\diamond_{#1}({#2})}
\newcommand{\jlct}[1]{\operatorname{lct}^\diamond({#1})}
\newcommand{\MIj}[1]{\J^\diamond\big({#1}\big)}

\newcommand{\km}[2]{\^k_{#1}({#2})} 
\newcommand{\am}[2]{\^a_{#1}({#2})} 
\newcommand{\Km}[2]{\^K_{{#1}/{#2}}}
\newcommand{\MIm}[1]{\^\J\big({#1}\big)}

\def\Jac{\frj}

\def\reg{\mathrm{reg}}
\def\red{\mathrm{red}}

\newcommand{\lcid}[1]{\frd_{#1}}
\newcommand{\nash}[1]{\frn_{#1}}
\newcommand{\cond}[1]{\frc_{#1}}

\newcommand{\grom}[1]{\om_{#1}^{\rm GR}}
\newcommand{\mom}[1]{\^\om_{#1}}

\DeclareMathOperator{\codim} {codim}

\DeclareMathOperator{\im} {Im}

\DeclareMathOperator{\Gr} {Gr}

\DeclareMathOperator{\Spec} {Spec}

\DeclareMathOperator{\val} {val}

\DeclareMathOperator{\Ex} {Ex}

\DeclareMathOperator{\ord} {ord}

\DeclareMathOperator{\Fitt} {Fitt}

\DeclareMathOperator{\Hom} {Hom}
\def\cHom{{\cH}om}

\begin{document}



\title    		 {Jacobian discrepancies and rational singularities} 
			
\author    {Tommaso de Fernex}
\address   {Department of Mathematics,
            University of Utah,
            155 South 1400 East,
            Salt Lake City, UT 84112, USA} 
\email     {defernex@math.utah.edu}

\author    {Roi Docampo}
\address   {Department of Mathematics,
            University of Utah,
            155 South 1400 East,
            Salt Lake City, UT 84112, USA} 
\email     {docampo@math.utah.edu}

\thanks    {The first author is partially supported by NSF CAREER Grant
            DMS-0847059.}
\thanks		{Compiled on \today. Filename \small\tt\jobname}



\begin{abstract}
Inspired by several works on jet schemes and motivic integration, we
consider an extension to singular varieties of the classical definition of
discrepancy for morphisms of smooth varieties. The resulting invariant,
which we call \emph{Jacobian discrepancy}, is closely related to the jet
schemes and the Nash blow-up of the variety. This notion leads to a
framework in which adjunction and inversion of adjunction hold in full
generality, and several consequences are drawn from these properties. The
main result of the paper is a formula measuring the gap between the
dualizing sheaf and the Grauert--Riemenschneider canonical sheaf of a
normal variety. As an application, we give characterizations for rational
and Du Bois singularities on normal Cohen--Macaulay varieties in terms of
Jacobian discrepancies. In the case when the canonical class of the variety
is $\Q$-Cartier, our result provides the necessary corrections
for the converses to hold in theorems of Elkik, of Kov\'acs, Schwede and
Smith, and of Koll\'ar and Kov\'acs on rational and Du Bois singularities.
\end{abstract}

\maketitle


\section{Introduction}
\label{s:intro}

The main result of the paper is a formula quantifying the difference
between the dualizing sheaf and the Grauert--Riemenschneider canonical
sheaf of a normal variety, and is stated below in Theorem~C. As a
motivation of this more technical result, we begin by first describing its
implications to the study of rational and Du~Bois singularities.

Rational singularities, first introduced and studied in dimension two as a
generalization of Du Val singularities \cite{Art66,Lip69-rat}, can be
thought as those singularities that do not contribute to the cohomology of
the structure sheaf of the variety. The connection with the singularities
in the minimal model program was discovered by Elkik \cite{Elk81}. Du Bois
singularities \cite{DB81,Ste83} form a wider, more mysterious class which
arises naturally from the point of view of Hodge theory and satisfies good
vanishing properties, see for instance \cite{GKKP11}. The link between Du
Bois singularities and log canonical singularities is a recent achievement
first established in the Cohen--Macaulay case in \cite{KSS10} and then,
unconditionally, in \cite{KK10}. For both classes of singularities, the
connection with the singularities in the minimal model program appears to
be unidirectional, as most rational and Du Bois singularities do not seem
at first to satisfy any reasonable condition from the point of view of
valuations and discrepancies.

As we shall see, there is in fact a deeper connection going the other way
around which provides characterizations of these singularities when the
variety is Cohen--Macaulay. 

The precise connection depends on the tension between two closely related
ideals attached to the singularities. The first one is the \emph{lci-defect
ideal} $\frd_X$ of $X$. This object is very natural from the point of view
of liaison theory and is related to the Nash transformation of $X$ with
respect to the dualizing sheaf $\om_X$. In concrete terms, $\lcid X$ is the
ideal generated by the equations of the residual intersections with all the
reduced, locally complete intersection schemes $V \supset X$ of the same
dimension. The second ideal, called the \emph{lci-defect ideal of level
$r$} of $X$ and denoted by $\lcid{r,X}$, is defined when the canonical
class $K_X$ is $\Q$-Cartier and depends on the integer $r$ such that $rK_X$
is Cartier. The two ideals agree when $r = 1$. In general, both ideals
vanish precisely on the locus where $X$ is not locally complete
intersection. 

In the following theorem we make the necessary correction for the converse
to hold in the aforementioned results of~\cite{Elk81, KSS10, KK10}.
Restricting to the Cohen--Macaulay case, the result yields a
characterization of rational and Du Bois singularities in terms of
discrepancies. In particular, since rational singularities are
Cohen--Macaulay, the result provides a discrepancy characterization of all
rational singularities.

\begin{thmA}
Let $X$ be a normal variety, and assume that $rK_X$ is Cartier for some
positive integer $r$.
\begin{enumerate}
\item 
If $X$ has rational singularities then the pair $(X,\lcid{r,X}^{1/r}\.\lcid X^{-1})$
is canonical.
\item 
If $X$ has Du Bois singularities then the pair $(X,\lcid{r,X}^{1/r}\.\lcid X^{-1})$ is
log canonical. 
\end{enumerate}
Moreover, the converse holds in both cases whenever $X$ is Cohen--Macaulay.
\end{thmA}

In either case, the Cohen--Macaulay condition is essential for the converse
to hold, see Example~\ref{eg:example}. If however the assumptions on the
singularities of the pair is strengthened by removing the contribution of
$\lcid X^{-1}$, then one obtains sufficient conditions for rational and Du
Bois singularities holding in a much more general setting. 

\medskip

We arrive at the above result by considering a quite different set of
questions that lead us to study an extension to arbitrary varieties of the
notion of discrepancy of a divisorial valuation over a smooth variety.

The candidate is an integer that simply measures the difference between the
Jacobian of the transformation (which gives the discrepancy in the smooth
case) and the Jacobian of the singularity. Specifically, given a resolution
of singularities $f \colon Y \to X$ of a complex variety $X$, we define the
\emph{Jacobian discrepancy} of a prime divisor $E \subset Y$ over $X$ to be
the integer
\[
\kj EX := \ord_E(\Jac_f) - \ord_E(\Jac_X),
\]
where $\Jac_f \subset \O_Y$ is the Jacobian ideal of the morphism $f$ and
$\Jac_X \subset \O_X$ is the Jacobian ideal of $X$. 

Jacobian discrepancies are closely related to \emph{Mather discrepancies},
where only the contribution coming from $\Jac_f$ is taken into account. The
latter are the main ingredient of the change-of-variables formula in
motivic integration \cite{DL99} and are determined by the Nash blow-up of
the variety. Geometric properties of Mather discrepancies were investigated
in \cite{dFEI08}, and the recent work of Ishii \cite{Ish11} is devoted to a
study of singularities from the point of view of Mather discrepancies. When
$X$ is $\Q$-Gorenstein (that is, $X$ is normal and its canonical class is
$\Q$-Cartier), the relationship between Jacobian discrepancies, Mather
discrepancies and usual discrepancies is implicit in the works
\cite{Kaw08,EM09,Eis10}.

Like Mather discrepancies, Jacobian discrepancies can be read off from the
jet schemes of the variety. In the following result, one should notice that
not just the topology but also the scheme structure of the spaces of jets
is relevant to this end. 

\begin{thmB}
For every prime divisor $E \subset Y$ over a variety $X$ we have
\[
\kj EX + 1 = 2n(m+1) - \dim_\C \big( TX_m|_{\e_{E,m}}\big) \fall m \ge 2\ord_E(\Jac_X),
\]
where $X_m$ is the $m$-th jet scheme of $X$, $TX_m$ is the tangent space to
$X_m$, and $\e_{E,m} \in X_m$ is the generic point of the image of the set
of $m$-th jets in $Y$ having order of contact one with $E$. 
\end{thmB}

Using Jacobian discrepancies, one can formulate in a completely natural way
a framework for singularities that runs parallel to the usual theory of
singularities of $\Q$-Gorenstein varieties considered in the minimal model
program. In particular, this leads to the notions of \emph{J-canonical} and
\emph{log J-canonical singularities}, of \emph{minimal log J-discrepancy},
and of \emph{log J-canonical threshold}. This framework moves in a
different direction with respect to the one proposed in \cite{dFH09}, where
the asymptotic nature of the theory on $\Q$-Gorenstein varieties is instead
taken into account. These invariants capture interesting new geometry of
the singularities and we believe that they deserve investigation. 

The theory is tailored to satisfy adjunction and inversion of adjunction,
see Proposition~\ref{p:adj} and Theorem~\ref{t:inv-adj}. The latter,
independently proven also by Ishii in \cite{Ish11}, extends to the setting
considered here the main theorems of \cite{EMY03, EM04, Kaw08, EM09}, and
has a number of consequences that were previously obtained in \cite{Mus01,
EM04, dFEM10} for normal, locally complete intersection varieties. These
include  the semi-continuity of minimal log J-discrepancies (see
Corollaries~\ref{c:semi-contI} and~\ref{c:semi-contII}, see also
\cite{Ish11}) and the fact that the set of all log J-canonical thresholds
in any fixed dimension satisfies the ascending chain condition (see
Corollary~\ref{c:ACC}). Theorem~B and inversion of adjunction also yield
characterizations of the singularities of a variety that involve jet
schemes and their tangent spaces, see Corollary~\ref{c:jet-sing}.

\medskip

The main application of the above viewpoint on singularities regards the
Grauert--Riemen\-schneider canonical sheaf of a variety $X$, defined by
$f_*\om_Y$ where $f \colon Y \to X$ is any resolution of singularities
\cite{GR70}. This is the natural generalization of the canonical line
bundle of a manifold from the point of view of Kodaira vanishing and plays,
for instance, a central role in Lipman's proof to resolution of
singularities in dimension two \cite{Lip78}. This sheaf agrees with the
dualizing sheaf if $X$ has rational singularities. In general there is an
inclusion $f_*\om_Y \subset \om_X$ and our aim is to provide a measure of
the gap. 

To this end, we introduce the natural generalization of multiplier ideals
to the above framework. Given a normal variety $X$, for any non-zero ideal
$\fra \subset \O_X$ and every real number $c$ we use Jacobian discrepancies
to define the \emph{Jacobian multiplier ideal} $\MIj{\fra^c}$. This is in
general a fractional ideal sheaf; it is an ideal sheaf if either $c \ge 0$
or $\fra$ is trivial in codimension one. This agrees with the usual
multiplier ideal if $X$ is locally complete intersection.

The core result of the paper is the following formula which expresses the
colon ideal of $f_*\om_Y$ by $\om_X$ as a Jacobian multiplier ideal. 

\begin{thmC}
If $f \colon Y \to X$ is a resolution of singularities of a normal variety,
then
\[
\big( f_*\om_Y : \om_X \big) = \MIj{\lcid X^{-1}},
\]
where the left-hand side is the colon of the sheaves viewed as
$\O_X$-modules.
\end{thmC}

In light of this theorem and its `twisted' version (stated below in
Theorem~\ref{t:can(X,Z)}), one can view the Grauert--Riemenschneider
vanishing theorem and its generalizations as Nadel-type vanishings for
Jacobian multiplier ideals. 

More importantly, using Kempf's characterization of rational singularities
on normal Cohen--Macaulay varieties and the analogous result on Du Bois
singularities established in \cite{KSS10}, we deduce from Theorem~C that,
given any normal variety $X$, 
\begin{enumerate}
\item 
if $X$ has rational singularities then the pair $(X,\lcid X^{-1})$
is J-canonical, 
\item 
if $X$ has Du Bois singularities then the pair $(X,\lcid X^{-1})$ is
log J-canonical,
\end{enumerate}
and in both cases the converse holds if $X$ is Cohen--Macaulay. This result
appears in the main body of the paper as Corollary~\ref{c:rat-DB}. When $X$
is $\Q$-Gorenstein, this implies Theorem~A.

As mentioned before, the Cohen--Macaulay condition is necessary for the
converse to hold. If however one imposes stronger conditions on
discrepancies, then a simple argument brought out to our attention by
\Mustata shows that the results of \cite{Kaw98,KK10}, in combination with
inversion of adjunction, imply that any variety with J-canonical (resp.,
log J-canonical) singularities has rational (resp., Du Bois) singularities
(see Theorem~\ref{t:rat-DB-inv-adj}). This last result is in fact
well-known to the specialists if the hypotheses on the singularities are
expressed in terms of the jet schemes of the variety.

\medskip

We work over the field of complex numbers. Unless otherwise stated, we use
the word \emph{scheme} to refer to a separated scheme of finite type over
$\C$. The word \emph{variety} will refer to an irreducible reduced scheme.

\subsection*{Acknowledgments} 

We would like to thank Lawrence Ein, Shihoko Ishii and Mircea \Mustata for
useful comments on a preliminary version of this paper, and for letting us
know about their preprint \cite{EIM} where they independently define and
study the same notion of multiplier ideals as the one considered in this
paper. We are also grateful to Ein and \Mustata for mentioning to us about
the recent work of Ishii \cite{Ish11}, and to Ishii for bringing to our
attention the paper \cite{Eis10}. We would like to thank Karl Schwede for
useful comments, for explaining to us about Du Bois singularities in the
case of affine cones, and for directing us to the reference \cite{Wat83}.
We also thank Sebastien Boucksom for useful comments. We are grateful to
the referee for the careful reading and many comments and corrections. The
article \cite{EM09} has been of invaluable help in the preparation of this
paper. The paper was completed while the first author was visiting the
\'Ecole Normale Sup\'erieure, and he wishes to thank the institution and
Olivier Debarre for providing support and a stimulating environment.

\section{Nash blow-up}
\label{s:Nash}

The notions of discrepancy that will be defined in the paper are naturally
related to the Nash blow-up. In this section we review the basic theory of
this blow-up, and explore a modification of the construction, the blow-up
of the dualizing sheaf.

\subsection{Jacobian ideals}

The \emph{Jacobian ideal sheaf} of a reduced scheme $X$, denoted $\Jac_X
\subset \O_X$, is the smallest non-zero Fitting ideal of the cotangent
sheaf $\Om_X$. If $X$ is of pure dimension $n$, then $\Jac_X = \Fitt^n
(\Om_X)$, the $n$-th Fitting ideal.

If $Y$ is a smooth scheme and $f \colon Y \to X$ is a morphism to a scheme
$X$, the \emph{Jacobian ideal} of $f$ is the $0$-th Fitting ideal $\Jac_f =
\Fitt^0(\Om_{Y/X})$ of $\Om_{Y/X}$. If $Y$ is equidimensional of dimension
$n$, then $\Om_Y^n$ is invertible and the image of the map induced at the
level of top differentials $f^*\Om_X^n \to \Om_Y^n$ can be written as
$\Jac_f \otimes \Om_Y^n$.

\subsection{The classical Nash blow-up} 

The \emph{Nash blow-up} of a reduced scheme $X$ of pure dimension $n$ is a
surjective morphism
\[
\n \colon \^X \to X
\]
satisfying the following universal property: a proper birational morphism
of schemes $f\colon Y \to X$ factors through $\n$ if and only if the sheaf
$f^*\Om_X$ has a locally free quotient of rank $n$. In general, if a
resolution $f\colon Y \to X$ factors through the Nash blow-up of $X$, then
the associated Jacobian ideal $\Jac_f$ is locally principal.

The Nash blow-up of $X$ is unique up to isomorphism, and can be constructed
by taking the restriction of the projection $\Gr(\Om_X, n) \to X$ to the
closure $\^X \subset \Gr(\Om_X, n)$ of the natural isomorphism $X_{\rm reg}
\to \Gr(\Om_{X_{\rm reg}},n)$. Here, for a coherent sheaf $\cF$ on $X$, we
denote by $\Gr(\cF, n)$ the Grassmannian of locally free quotients of $\cF$
of rank $n$. If $X$ is embedded in a smooth variety $M$, the natural
quotient $\Om_M|_X \to \Om_X$ induces and inclusion $i\colon \Gr(\Om_X, n)
\hookrightarrow \Gr(\Om_M, n)$, and $\^X$ is the closure of the natural
embedding of $X_{\rm reg}$ in $\Gr(\Om_M, n)$ given by $i$. Alternatively,
using the Pl\"ucker embedding $\Gr(\Om_X,n) \subset \P(\Om_X^n)$, one can
also view the Nash blow-up $\^X$ inside $\P(\Om_X^n)$ as the closure of the
natural isomorphism $X_{\rm reg} \cong \P(\Om_{X_{\rm reg}}^n)$. Denoting
for short
\[
\mom X := \Om_X^n/{\rm torsion},
\]
$\^X$ can also be viewed as the closure of $X_{\reg}$ in $\P(\mom X)$ since
the latter is closed in $\P(\Om_X^n)$ and contains $\P(\Om_{X_{\rm reg}}^n)$.

\begin{rmk}
Since $X$ is reduced, $\om_X$ is torsion free (cf.\ Proposition~(2.8) of
\cite{AK70}) and the canonical map $\Om_X^n \to \om_X$ (cf.\
Proposition~9.1 of \cite{EM09}) factors through an inclusion $\mom X \inj
\om_X$.
\end{rmk}

Note that the tautological line bundle $\O_{\P(\Om_X^n)}(1)$ of
$\P(\Om_X^n)$ restricts to the tautological line bundle $\O_{\P(\mom
X)}(1)$ of $\P(\mom X)$. Following the terminology introduced in
\cite{dFEI08}, we refer to the restriction $\O_{\^X}(1) := \O_{\P(\mom
X)}(1)|_{\^X}$ the \emph{Mather canonical line bundle} of $X$, and to any
Cartier divisor $\^K_X$ on $\^X$ such that $\O_{\^X}(\^K_X) \cong
\O_{\^X}(1)$ a \emph{Mather canonical divisor} of $X$.

\begin{rmk}
The above terminology is motivated by the relationship with Mather-Chern
classes, see Remark~1.5 of \cite{dFEI08} for a discussion. Note that in
\cite{dFEI08} the symbol $\^K_X$ was used to denote the Mather canonical
line bundle.
\end{rmk}

\begin{rmk}
\label{r:O(1)}
If $X \subset X'$ is the inclusion between two reduced equidimensional
schemes of the same dimension and we denote by $\n \colon \^X \to X$ and
$\n' \colon \^X' \to X'$ the respective Nash blow-ups, then $\^X \subset
\^X'$, $\n = \n'|_{\^X}$ and $\O_{\^X}(1) = \O_{\^X'}(1)|_{\^X}$.
\end{rmk}

The Nash blow-up naturally relates to the blow-up of the Jacobian ideal of
the scheme. Lemma~1 of \cite{Lip69} implies that for any reduced
equidimensional scheme $X$, the blow-up of the Jacobian ideal of $X$
factors through the Nash blow-up of $X$; see also \cite{OZ91} for a
detailed discussion of this property. A direct computation in local
coordinates shows that when $X$ is locally complete intersection the Nash
blow-up of $X$ is isomorphic to the blow-up of the Jacobian ideal of $X$,
see \cite{Nob75,OZ91}. One can use Remark~\ref{r:O(1)} to see that in
general, if $X$ is a reduced equidimensional scheme and $V \supset X$ is a
reduced, locally complete intersection scheme of the same dimension, then
the Nash blow-up $\n \colon \^X \to X$ is isomorphic to the blow-up of the
ideal $\Jac_V|_X$. Since the natural map $\Om_X^n \to \om_V|_X$ has image
\[
\im\big[\Om_X^n \to \om_V|_X\big] = \Jac_V|_X \otimes \om_V|_X
\]
(see for instance \cite{OZ91,EM09}), this fact also follows from the
discussion given in the following subsection. 

\begin{prop}
\label{p:taut-bundles} 
Let $X$ be a reduced equidimensional scheme. Then, for every reduced,
locally complete intersection scheme $V \supset X$ of the same dimension,
there is a natural isomorphism
\[
\Jac_V|_X\.\O_{\^X} \cong \O_{\^X}(1) \otimes \n^* (\om_V^{-1}|_X).
\]
\end{prop}

\begin{proof}
We view $\n \colon \^X \to X$ as the blow-up of $\Jac_V|_X$. Since by
generic smoothness the kernel of the natural map $\Om_X^n \to \om_V|_X$ is
the torsion of $\Om_X^n$, there is an isomorphism 
\[
\P(\mom X) =
\P(\Jac_V|_X\otimes\,\om_V|_X) \stackrel{\a}\cong \P(\Jac_V|_X).
\] 
This implies that
\[
\O_{\P(\mom X)}(1) \cong \a^*\O_{\P(\Jac_V|_X)}(1) \otimes \p^*\om_V|_X
\]
where $\p \colon \P(\mom X) \to X$ is the projection map. On the other
hand, by the universal property of $\P(\Jac_V|_X)$, the blow-up $\^X \to X$
of $\Jac_V|_X$ factors through $\P(\Jac_V|_X) \to X$, and
$\O_{\P(\Jac_V|_X)}(1)$ pulls back to $\Jac_V|_X\.\,\O_{\^X}$. Since the
map $\^X \to \P(\Jac_V|_X)$ is an isomorphism onto its image, we obtain
$\O_{\^X}(1) \cong (\Jac_V|_X\.\O_{\^X}) \otimes \n^* \om_V|_X$. 
\end{proof}

\subsection{The blow-up of the dualizing sheaf}

Let $X$ be a reduced equidimensional scheme. We denote by $\C(X)$ the sheaf
of rational functions on $X$. Following the general definition given in
\cite{Kle79}, this sheaf is given by the push forward of the restriction of
$\O_X$ to the associated primes, which in our case are just the generic
points of the irreducible components of $X$.

In the construction of the Nash blow-up, one can replace the sheaf of
differentials by any coherent sheaf $\cF$ which is locally free of some
rank $r$ in an open dense subset of $X$. This idea was developed in detail
in \cite{OZ91}, where it received the name of \emph{Nash transformation of
$X$ relative to $\cF$}.

The main result from \cite{OZ91} that we will use is a description of the
ideal whose blow-up gives the Nash transformation: it is the image of
composition
\[
\wedge^r \cF 
\longrightarrow
\wedge^r \cF \otimes_{\O_X} \C(X) 
\stackrel\cong\longrightarrow
\C(X),
\]
where the second arrow is induced by any trivialization of $\wedge^r \cF$
at each of the generic points of $X$. Notice that different choices of
trivialization give different ideals, but they differ by a Cartier divisor,
so they give rise to the same blow-up. Also, for an arbitrary
trivialization the ideal is possibly fractional, but there is always a
choice that clears all the denominators.

For example, in the case of the classical Nash blow-up one is led to
consider the natural sequence 
\[
\mom X
\longrightarrow
\om_X
\longrightarrow
\om_V|_X
\longrightarrow
\Om^n_{\C(X)/\C},
\]
where $V$ is some reduced locally complete intersection $n$-dimensional
scheme containing $X$. Since $\om_V|_X$ is a line bundle, one can use it to
trivialize $\Om^n_{\C(X)/\C}$, and we see that the classical Nash blow-up
is given by the ideal $\fra$ for which the image of $\mom X$ in $\om_V|_X$
is $\fra \otimes \om_V|_X$, namely, the ideal $\Jac_V|_X$. 

For our purposes, it will be interesting to consider a second Nash
transformation, this time relative to the dualizing sheaf. It is a proper
birational morphism
\[
\m \colon \~X \longrightarrow X,
\]
which is universal among those proper birational morphisms for which the
pull back of $\om_X$ admits a line bundle quotient. Analogous to the case
of the classical Nash blow-up, $\~X$ can be constructed as the closure in
$\P(\om_X)$ of $X_{\rm reg} \simeq \P(\om_{X_{\rm reg}})$. 

Our interest in $\~X$ arises from the ideals whose blow-up gives $\m$.
After picking some locally complete intersection $V$ as above, one sees
that $\m$ is the blow-up of $X$ along the ideal $\lcid{X,V}$ for which the
image of the inclusion of $\om_X$ inside $\om_V|_X$ is given by
\[
\im\big[\om_X \to \om_V|_X\big] = \lcid{X,V} \otimes \om_V|_X.
\]
Note that $\m$ is an isomorphism if and only if $\om_X$ is invertible. One
implication is obvious, and conversely, if $\m$ is an isomorphism then
$\lcid{X,V}$ is locally principal and hence $\om_X = \lcid{X,V} \otimes
\om_V|_X$ is invertible.

Different choices of $V$ give different ideals $\lcid{X,V}$, but their
blow-up is always $\~X$. These ideals are easy to describe (cf.\
\cite{OZ91,EM09}). One embeds $V$ in a smooth variety $M$, and considers
the ideals $I_X$ and $I_V$ of $X$ and $V$ in $M$. Then, as $\O_V$-modules,
one has
\[
\om_X \otimes \om_V^{-1} =
\cH om_{\O_V}(\O_X, \O_V) = 
(I_V : I_X) / I_V,
\]
and therefore
\[
\lcid{X,V} = 
((I_V : I_X) + I_X)/ I_X.
\]
In other words, if we write $V = X \cup X'$, where $X'$ is the residual
part of $V$ with respect to $X$ (given by the ideal $(I_V:I_X)$), then
$\lcid{X,V}$ is the ideal defining the intersection $X \cap X'$ in $X$. We
may think of $\lcid{X,V}$ as giving the \emph{residual intersection} of $V$
with $X$.

As $V$ varies, so does the residual intersection $\lcid{X,V}$. Thinking of
this collection of intersections as a linear series in $X$, it is natural
to consider its base locus, which is clearly supported on the points where
$X$ is not locally complete intersection. This motivates the next
definition. 

\begin{defi}
The \emph{lci-defect ideal} of $X$ is defined to be
\[
\lcid X := \sum_{V} \lcid{X,V},
\]
where the sum is taken over all reduced, locally complete intersection
schemes $V \supset X$ of the same dimension. 
\end{defi}

In the case of the classical Nash blow-up, the analogous object is the
Jacobian ideal $\Jac_X$, which is spanned by the ideals $\Jac_V|_X$ as $V$
varies in any fixed embedding $M$ of $X$. If we restrict the sum in the
above definition to those schemes $V$ varying in one fixed embedding $M$,
the resulting ideal, which we denote by $\lcid{X/M}$, depends a priori on
the embedding. Its integral closure however does not depend on $M$.

\begin{prop}
\label{p:indep-M}
For any fixed embedding $X \subset M$, we have $\ov{\lcid {X/M}} =
\ov{\lcid X}$.
\end{prop}

\begin{proof}
Since 
\[
\lcid X = \sum_{M \supset X} \lcid{X/M},
\]
it suffices to prove that the integral closure of $\lcid{X/M}$ is
independent of the embedding. Fix any embedding $X \subset M$. If
$\O_{\~X}(1)$ denotes the tautological quotient of $\m^*\om_X$ associated
to the Nash transformation, then 
\[
\im\big[\m^*\mom X \to \m^*\om_X \to \O_{\~X}(1)\big] 
= \frn\otimes \O_{\~X}(1)
\]
for some ideal $\frn \subset \O_{\~X}$. On the other hand, we have
\[
\im\big[\m^*\mom X \to \m^*\om_V|_X\big] 
= (\Jac_V|_X\.\O_{\~X})\otimes \m^*(\om_V|_X).
\]
By the same arguments of the proof of Proposition~\ref{p:taut-bundles}
there is a natural isomorphism 
\[
(\lcid{X,V}\.\O_{\~X})\otimes \m^*(\om_V|_X) \cong \O_{\~X}(1),
\]
and we see that $\frn\.\big(\lcid{X,V}\.\O_{\~X}\big) =
\Jac_V|_X\.\O_{\~X}$. By taking the sum as $V$ varies, we obtain
$\frn\.\big(\lcid {X/M}\.\O_{\~X}\big) = \Jac_X\.\O_{\~X}$. This proves
that the integral closure of $\lcid {X/M}\.\O_{\~X}$ is independent of $M$. 
\end{proof}

\begin{rmk}
For the purpose of this paper, the integral closure of $\lcid X$ is the
only thing we need to control, and thus with slight abuse of notation one
can always pretend that $\lcid X$ is determined from any embedding of $X$. 
\end{rmk}

\subsection{The \QQ-Gorenstein case}

When $X$ is $\Q$-Gorenstein, one can exploit this property to give another
scheme structure to the locus where $X$ is not locally complete
intersection, as it is explained in \cite{EM09}. We review here this
alternative theory, as it will be useful later to compare the new notions
of discrepancy with the classical one.

Recall that a variety $X$ is said to be \emph{$\Q$-Gorenstein} if it is
normal and its canonical class $K_X$ is $\Q$-Cartier. Assume $X$ is
$\Q$-Gorenstein, and let $r$ be a positive integer such that $rK_X$ is
Cartier. For any reduced, locally complete intersection scheme $V \supset
X$ of the same dimension, we consider the ideal $\lcid{r,X,V} \subset \O_X$
such that the image of the natural map from $\O_X(rK_X)$ to
$(\om_V|_X)^{\otimes r}$ is given by 
\[
\im\big[\O_X(rK_X) \to (\om_V|_X)^{\otimes r}\big] 
= \lcid{r,X,V}\otimes (\om_V|_X)^{\otimes r}.
\]
Note that $\lcid{r,X,V}$ is a locally principal ideal since $\O_X(rK_X)$ is
a line bundle. 

\begin{defi}
With the above assumptions, we define the \emph{lci-defect ideal of level
$r$} of $X$ to be
\[
\lcid {r,X} := \sum_{V} \lcid{r,X,V},
\]
where the sum is taken over all $V$ as above.
\end{defi}

Note also that, like $\lcid X$, the ideal $\lcid {r,X}$ vanishes exactly
where $X$ is not locally complete intersection. If we fix an embedding $X
\subset M$ then the ideal $\lcid{r,X/M}$ obtained by restricting the sum in
the above definition to schemes $V \subset M$ depends a priori on the
embedding. 

To better understand these ideals, for every $V$ as above we consider the
natural sequence
\begin{equation}
\label{eq:om^r}
\^\om^{\otimes r}_X
\longrightarrow
\om^{\otimes r}_X
\longrightarrow
\O_X(rK_X)
\longrightarrow
(\om_V|_X)^{\otimes r}.
\end{equation}
Using again the fact that $\O_X(rK_X)$ is a line bundle, we see that the
image of the canonical map from $\^\om_X^{\otimes r}$ to $\O_X(rK_X)$ is
given by 
\[
\im\big[\^\om_X^{\otimes r} \to \O_X(rK_X)\big] =
\nash{r,X} \otimes \O_X(rK_X)
\]
for some ideal sheaf $\nash{r,X} \subset \O_X$. 

\begin{defi}
In accordance with the terminology introduced in \cite{EM09}, we call
$\nash{r,X}$ the \emph{Nash ideal of level $r$} of X. 
\end{defi}

\begin{prop}
\label{p:lcid_r-jac^r}
With the above notation, there are inclusions
\[
\nash{r,X} \cdot \lcid{r,X/M} \subset 
\nash{r,X} \cdot \lcid{r,X} \subset \Jac_X^r
\]
which induce identities on integral closures. In particular, for every
embedding $X \subset M$ we have $\ov{\lcid{r,X/M}} = \ov{\lcid{r,X}}$.
\end{prop}

\begin{proof}
Since the image of $\^\om_X$ in $\om_V|_X$ is given by $\Jac_V|_X$, we have
that $\nash{r,X} \cdot \lcid{r,X,V} = (\Jac_V|_X)^r$, and hence
\[
\sum_{V \subset M} (\Jac_V|_X)^r = \nash{r,X} \cdot \lcid{r,X/M} \subset
\nash{r,X} \cdot \lcid{r,X} = \sum_{V} (\Jac_V|_X)^r \subset \Jac_X^r,
\]
and both inclusions become equality once integral closures are taken. The
last assertion is a direct consequence of the first part. 
\end{proof}

\begin{rmk}
\label{r:lcid-r-frb}
With the above notation, the ideal $\lcid{r,X}$ has the same integral
closure as the colon ideals $J'_r := \(\, \Jac_X^r : \nash{r,X} \,\)$ and
$J_r := \(\, \ov{\Jac_X^r} : \nash{r,X} \,\)$. First notice that there is a
chain of inclusions $\lcid{r,X} \subset J'_r \subset J_r$, the first one
following by Proposition~\ref{p:lcid_r-jac^r} and the latter by the
inclusion $\Jac_X^r$ in its integral closure. Therefore it suffices to
check that $\ov{\lcid{r,X}} = \ov{J_r}$. This follows by combining
Proposition~\ref{p:lcid_r-jac^r}, which gives $\ov{\nash{r,X}\.\lcid{r,X}}
= \ov{\Jac_X^r}$, with Proposition~9.4 of \cite{EM09}, which gives
$\ov{\nash{r,X}\.J_r} = \ov{\Jac_X^r}$. In \cite{EM09}, the ideal $J_r$ is
chosen as a measure of the locus where $X$ is not locally complete
intersection, and the scheme it defines is called the \emph{non-lci
subscheme of level $r$} of $X$. For the purpose of this paper, there is no
significant difference between the two ideals as we only need to take into
consideration the valuative contribution of these ideals, which are
equivalent. 
\end{rmk}

\begin{prop}
\label{p:lcid-r}
With the above notation, we have $\lcid X^r \subset \ov{\lcid{r,X}}$.
\end{prop}

\begin{proof}
By the definitions and the sequence \eqref{eq:om^r} we have an inclusion
$\lcid{X,V}^r \subset \lcid{r,X,V}$, and hence $\sum _V \lcid{X,V}^r
\subset \lcid{r,X}$. One concludes by observing that $\lcid X^r$ is
contained in the integral closure of $\sum _V \lcid{X,V}^r$. 
\end{proof}

\begin{rmk}\label{r:lcid-1}
If $\om_X$ is invertible then we can take $r=1$, and it follows from the
definitions that $\lcid X = \lcid{1,X}$. In general however the inclusion
$\ov{\lcid X^r} \subset \ov{\lcid{r,X}}$ might be strict (cf.\
Remark~\ref{r:strict} below).  
\end{rmk}

\section{Discrepancies}
\label{s:discr}

In this section we introduce the main invariants that we will study
throughout the paper.

\subsection{Mather and Jacobian discrepancies}

We focus on the case of a reduced equidimensional scheme $X$. To fix
terminology, by a \emph{resolution} of a reduced equidimensional scheme $X$
we intend a morphism $f \colon Y \to X$ from a smooth scheme $Y$ that
restricts to a proper birational map over each irreducible component of $X$
and such that every irreducible component of $Y$ dominates  an irreducible
component of $X$. In particular, if $f \colon Y \to X$ is a resolution and
$X$ has irreducible components $X_i$, then $Y$ decomposes as a disjoint
union of smooth varieties $Y_i$ and $f$ restricts to resolutions $f_i
\colon Y_i \to X_i$.

We say that $E$ is a prime divisor \emph{over} $X$ if $E$ is a prime
divisor on some resolution $f\colon Y \to X$, and that it is
\emph{exceptional} if $f$ is not an isomorphism at the generic point of
$E$. The \emph{center} $c_X(E)$ of $E$ in $X$ is the generic point of the
image of $E$ in $X$. If $Y_i$ is the component of $Y$ containing $E$ and
$f_i \colon Y_i \to X_i$ is the induced resolution of the corresponding
irreducible component of $X$, then associated to $E$ one defines the
divisorial valuation $\ord_E$ on the function field $\C(X_i)$: for any
nonzero element $\f \in \C(X_i)$, $\ord_E(\f)$ is the order of vanishing
(or pole) of $f_i^*(\f)$ at the generic point of $E \subset Y_i$. 

If $\fra \subset \O_X$ is an ideal sheaf on $X$, then we denote
$\ord_E(\fra) := \ord_E(\fra|_{X_i})$. Equivalently, we have $\ord_E(\fra)
= \ord_E(\fra\.\O_Y)$ where the right-hand-side denotes the integer $a$ for
which the image of $\fra\.\O_Y$ in the discrete valuation ring $\O_{Y,E}$
is the $a$-th power of the maximal ideal. Note that $\ord_E(\fra) = \infty$
if and only if $\fra$ vanishes identically along the irreducible component
$X_i$ of $X$ dominated by the component of $Y$ containing $E$.

\begin{defi}
Let $X$ be a reduced equidimensional scheme. Given a resolution $f \colon Y
\to X$, if $E$ is a prime divisor over $X$, the \emph{Mather discrepancy}
and the \emph{Jacobian discrepancy} of $E$ over $X$ are respectively
defined by
\[
\km EX := \ord_E(\Jac_f)
\and
\kj EX := \ord_E(\Jac_f) - \ord_E(\Jac_X).
\]
The \emph{relative Mather canonical divisor} and the \emph{relative
Jacobian canonical divisor} of $f$ are, respectively, 
\[
\Km YX := \sum_{E \subset Y} \km EX\. E \and
\Kj YX := \sum_{E \subset Y} \kj EX\. E,
\]
where the sum runs over all prime divisors $E$ on $Y$. 
\end{defi}

\begin{rmk}
If $X$ is smooth, then $\km EX = \kj EX = \ku EX$, the usual discrepancy of
$E$ over $X$, and $\Km YX = \Kj YX = \Ku YX$, the usual relative canonical
divisor.
\end{rmk}

If $X$ is a reduced equidimensional scheme and $f\colon Y \to X$ is a
resolution factoring through the Nash blow-up of $X$, then
\[
\Jac_f\.\O_Y = \O_Y(- \Km YX),
\]
and $\Km YX$ is the unique $f$-exceptional divisor linearly equivalent to
$K_Y - \^f^*\^K_X$ for a choice of a canonical divisor $K_Y$ on $Y$ and of
a Mather canonical divisor $\^K_X$ (see Proposition~1.7 of \cite{dFEI08}).
Furthermore, if $f \colon Y \to X$ factors through the blow-up of $\Jac_X$,
and we write $\Jac_X \.\O_Y = \O_Y(-B)$ for some effective divisor $B$ on
$Y$, then
\[
\Kj YX = \Km YX - B.
\]
One deduces the following property. 

\begin{prop}
\label{p:K-compos}
Let $f \colon Y \to X$ be a resolution of a reduced equidimensional scheme
$X$ that factors through the blow-up of the Jacobian ideal\/ $\Jac_X$. If
$f'\colon Y' \to X$ is another resolution factoring through $f$ via a
morphism $h \colon Y' \to Y$, then 
\[ 
\Km {Y'}X = \Ku {Y'}Y + h^*\Km YX \and
\Kj {Y'}X = \Ku {Y'}Y + h^*\Kj YX.
\] 
\end{prop}

\subsection{Discrepancies over \QQ-Gorenstein varieties}

Suppose that $X$ a $\Q$-Gorenstein variety. In this case one defines the
relative canonical divisor of a resolution $f \colon Y \to X$ as the
$\Q$-divisor $\Ku YX= K_Y - f^*K_X$ where $K_Y$ is a canonical divisor on
$Y$, $K_X = f_*K_Y$, and $f^*K_X$ is defined as the pull-back of a
$\Q$-Cartier divisor. If $E$ is a prime divisor on $Y$, we denote by $\ku
EX := \ord_E(\Ku YX)$. This is the usual discrepancy of $E$ over $X$.

The relation between Mather and Jacobian discrepancies and the usual
discrepancies defined for this class of varieties is implicit in the works
\cite{Kaw08,EM09,Eis10}, and is made explicit in the following statement. 

\begin{prop}
\label{p:Q-Gor}
Let $X$ be a $\Q$-Gorenstein variety. Let $r$ be any positive integer such
that $rK_X$ is Cartier, let $\nash{r,X}$ be the Nash ideal of level $r$ of
$X$, and let $\lcid{r,X}$ be the lci-defect ideal of level $r$ of $X$. Then
for every prime divisor $E$ over $X$ we have
\[
\km EX = \ku EX + \tfrac 1r \ord_E(\nash{r,X}) \and
\kj EX = \ku EX - \tfrac 1r \ord_E(\lcid{r,X}).
\]
\end{prop}

\begin{proof}
Let $f \colon Y \to X$ be a log resolution of $X$ such that $E$ is a
divisor on $Y$ and both $\nash{r,X}\.\O_Y$ and $\Jac_f$ are locally
principal. Fix a canonical divisor $K_Y$ on $Y$ such that $f_*K_Y = K_X$,
and let $D$ be an effective divisor on $Y$ such that $rK_{Y/X} + D \ge 0$.
Consider the composition $\g$ of maps
\[
f^*(\Om_X^n)^{\otimes r} \stackrel\a\longrightarrow
\O_Y(f^*(rK_X)) 
\stackrel\b\longrightarrow
\O_Y(rK_Y+D)
\stackrel\cong\longrightarrow
(\Om_Y^n)^{\otimes r} \otimes \O_Y(D),
\]
where $\b$ is is induced by a global section of $\O_Y(rK_{Y/X} + D)$.
The maps $\a$, $\b$ and $\g$ have images, respectively, 
\begin{align*}
\im(\a) &= (\nash{r,X}\.\O_Y) \otimes \O_Y(f^*(rK_X))\\
\im(\b) &= \O_Y(-rK_{Y/X} - D)\otimes\big((\Om_Y^n)^{\otimes r} \otimes \O_Y(D)\big)\\
\im(\g) &= \big(\Jac_f^r \otimes\O_Y(- D)\big)
\otimes\big((\Om_Y^n)^{\otimes r} \otimes \O_Y(D)\big). 
\end{align*}
By comparing images, we see that 
\[
\Jac_f^r = \nash{r,X} \.\O_Y(-rK_{Y/X}).
\] 
Since the ideals $\Jac_X^r$ and $\nash{r,X} \. \lcid{r,X}$ have the same
integral closure (see Proposition~\ref{p:lcid_r-jac^r}), we conclude that
\begin{align*}
\ord_E(\Jac_f)
&= \ord_E(\Ku YX) + \tfrac 1r \ord_E(\nash{r,X})\\
&= \ord_E(\Ku YX) + \ord_E(\Jac_X) - \tfrac 1r \ord_E(\lcid{r,X}),
\end{align*}
and both formulas follow. 
\end{proof}

\begin{cor}
\label{c:lci-K_Y/X}
If $X$ is a locally complete intersection variety and $f\colon Y \to X$ is
a resolution of singularities, then $\Kj YX = \Ku YX$.
\end{cor}

\begin{proof}
By Proposition~\ref{p:Q-Gor}, since $\lcid{r,X}$ is trivial if $X$ is
locally complete intersection.
\end{proof}

\section{Singularities}
\label{s:sing}

This section is devoted to the study of singularities of pairs from the
point of view of Jacobian discrepancies. We refer to \cite{KM98} for an
introduction to singularities of pairs in the usual setting.

\subsection{Pairs and singularities}

Throughout the paper, a \emph{pair} $(X,\frA)$ will always consists of a
reduced equidimensional scheme $X$ and a \emph{proper $\R$-ideal} $\frA =
\prod_k \fra^{c_k}$ of $X$, namely, a finite formal product, with real
exponents $c_k$, of ideal sheaves $\fra_k \subset \O_X$ such that
$\fra_k|_{X_i} \ne (0)$ on every irreducible component $X_i$ of $X$. The
$\R$-ideal $\frA$, and the pair itself, are said to be \emph{effective} if
$c_k \ge 0$ for all $k$. They are said to be \emph{effective in codimension
one} if $c_k \ge 0$ for all $k$ such that $\dim Z(\fra_k) = \dim X - 1$,
where in general $Z(\fra) \subset X$ denotes the subscheme defined by an
ideal sheaf $\fra \subset \O_X$. The \emph{vanishing locus} of $\frA$ is
the union of the supports of the $Z(\fra_k)$. 

For any map $g \colon X' \to X$ we denote $\frA\.\O_{X'} := \prod_k
(\fra_k\.\O_{X'})^{c_k}$. If $g$ is an inclusion, then we also denote
$\frA|_{X'} := \frA\.\O_{X'}$. For any real number $\l$, we denote $\frA^\l
:= \prod_k \fra_k^{\l c_k}$. If $E$ is a prime divisor over $X$, then we
denote $\ord_E(\frA) := \sum_k c_k \ord_E(\fra_k)$.

\begin{defi}
Let $(X,\frA)$ be a pair, and let $E$ be a prime divisor over $X$. The
\emph{log Jacobian discrepancy} (or \emph{log J-discrepancy}) of $E$ over
$(X,\frA)$ is the number
\[
\aj E{X,\frA} := \kj EX + 1 - \ord_E(\frA).
\]
The pair $(X,\frA)$ is said to be \emph{log J-canonical} (resp., \emph{log
J-terminal}) if $\aj E{X,\frA} \ge 0$ (resp., $\aj E{X,\frA} > 0$) for all
prime divisors $E$ over $X$. The pair $(X,\frA)$ is said to be
\emph{J-canonical} (resp., \emph{J-terminal}) if $\aj E{X,\frA} \ge 1$
(resp., $\aj E{X,\frA} > 1$) for all exceptional prime divisors $E$ over
$X$. If $\frA = \O_X$, then we just drop it from the notation. In
particular, we say that $X$ is \emph{J-canonical} or \emph{log J-canonical}
if so is the pair $(X,\O_X)$. 
\end{defi}

\begin{rmk}
If one defines the \emph{log Mather discrepancy} of $E$ over $(X,Z)$ to be
$\am E{X,\frA} := \km EX + 1 - \ord_E(\frA)$, then $\aj E{X,\frA} = \am
E{X,\frA\.\Jac_X}$. This invariant is studied in \cite{Ish11}.
\end{rmk}

\begin{defi}
Let $(X,\frA)$ be a pair. If $(X,\frA)$ is log J-canonical, then the
\emph{log J-canonical threshold} of $(X,\frA)$ is 
\[
\jlct{\frA} := 
\sup \{\,\l\ge 0 \mid \text{$(X,\frA^\l)$ is log J-canonical}\,\} 
\in [0,\infty].
\]
Here we set $\jlct{\frA} = 0$ if $\fra_k|_{X_i} = (0)$ for some $i$ and
some $k$ such that $c_k > 0$. For any Grothendieck point $\e \in X$, the
\emph{minimal log J-discrepancy} of $(X,T)$ at $\e$ is
\[
\jmld \e{X,\frA} := \inf_{c_X(E) =\e} \aj E{X,\frA}.
\]
If $\e_{X_i}$ is the generic point of an irreducible component $X_i$ of
$X$, then we set by definition $\jmld {\e_{X_i}}{X,\frA} =0$. If $T \subset
X$ is a closed set, then we denote
\[
\jmld T{X,\frA} := \inf_{\e \in T} \jmld \e{X,\frA}.
\]
\end{defi}

\begin{rmk}
If $X$ is locally complete intersection, then the above invariants agree
with the usual analogous invariants: $\aj E{X,\frA} = \au E{X,\frA}$,
$\jlct{X,\frA} = \lct{X,\frA}$, and $\jmld T{X,\frA} = \mld T{X,\frA}$, the
usual log discrepancy, log canonical threshold, and minimal log
discrepancy.
\end{rmk}

\begin{rmk}
Let $(X,\frA)$ be any pair as above. Let $X' \to X$ be the normalization,
and let $\cond X := \cHom_{\O_X}(\O_{X'},\O_X)$ be the conductor ideal;
this is the largest ideal sheaf in $\O_X$ which is also an ideal sheaf in
$\O_{X'}$. By Theorem~2.4 of \cite{Yas07}, there is an inclusion $\Jac_X
\subset \cond X$, and hence for every prime divisor $E$ over $X$ one has
\[
\am E{X,\frA\.\Jac_X} \le \am E{X',\frA'\.\cond X}
\]
where $\frA' := \frA\.\O_{X'}$. In the special case where $X$ is a simple
normal crossing divisor in a smooth variety $M$ of dimension $n+1$, we
actually have an equality $\Jac_X = \cond X$, and thus if $X_i$ is the
irreducible component of $X$ over which $E$ lies then
\[
\aj E{X,\frA} = \au E{X_i,\frA|_{X_i}\.\cond X|_{X_i}}.
\]
\end{rmk}

\begin{defi}
A \emph{log resolution} of a pair $(X,\frA)$ consists of a log resolution
$f \colon Y \to X$ of $X$ such that $\fra_k\.\O_Y$ is locally principal for
every $k$ and the union of their vanishing loci, together with the
vanishing locus of $\Jac_X\.\,\O_Y$ and the exceptional locus $\Ex(f)$,
form a simple normal crossing divisor. 
\end{defi}

The existence of log resolutions follows from Hironaka's resolution of
singularities \cite{Hir64}. Note that the above definition of log
resolution differs slightly from the usual one (it is more restrictive) in
that we also impose conditions on the pull-back of the Jacobian ideal of
$X$. Note that, according to our definition, every log resolution factors
through the blow-up of the Jacobian ideal of $X$, and thus through the Nash
blow-up of $X$.

Minimal log J-discrepancies of a pair $(X,\frA)$ are trivial at the generic
point of $X$ and are easy to compute at points of codimension one.  Note
also that since the normalization of $X$ gives a log resolution of
$(X,\frA)$ in codimension one, minimal log J-discrepancies in codimension
one are computed on any log resolution of the pair.

\begin{prop}
\label{p:Jmld}
Let $(X,\frA)$ be a pair as above, and let $T \subset X$ be a closed subset
of codimension $\ge 1$.
\begin{enumerate}
\item
If $(X,\frA)$ is log J-canonical in a neighborhood of $T$, then $\jmld
T{X,\frA}$ is realized on any given log resolution $f \colon Y \to X$ of
$(X,\frA\.\I_T)$ as the log J-discrepancy $\aj E{X,\frA}$ of some
prime divisor $E$ on $Y$.
\item
If $(X,\frA)$ is not log J-canonical in any neighborhood of $T$ and
$\dim X \ge 2$, then $\jmld T{X,\frA} = - \infty$.
\end{enumerate}
\end{prop}

\begin{proof}
Using Proposition~\ref{p:K-compos}, the proof is the same as the one of the
analogous properties for usual minimal log discrepancies on $\Q$-Gorenstein
varieties.
\end{proof}

\subsection{Adjunction}

Let $X$ be a reduced equidimensional scheme, embedded in a smooth variety
$M$. Let $\frA$ be a proper $\R$-ideal on $M$ such that $\frA|_X$ is a
proper $\R$-ideal on $X$. We denote by $\Jac_{X,M} \subset \O_M$ the ideal
defining the scheme $Z(\Jac_X)$ viewed as a subscheme of $M$. We will use
the following version of embedded resolution. 

\begin{defi}
An \emph{embedded log resolution}
of $(X,\frA|_X)$ in $(M,\frA)$ consists of a log resolution $g \colon N
\to M$ of $(M,\frA\.\Jac_{X,M})$ satisfying the following properties:
\begin{enumerate}
\item[(i)]
$g$ is is an isomorphism at the generic point of each irreducible component
$X_i$ of $X$;
\item[(ii)]
the restriction of $g$ to the proper transform $Y$ of $X$ gives a log
resolution $f \colon Y \to X$ of $(X,\frA|_X)$;
\item[(iii)]
$Y$ is transverse to the simple normal crossing divisor given by the
union of $\Ex(g)$ and the vanishing locus of $(\frA\.\Jac_{X,M})\.\O_N$. 
\end{enumerate}
If $\frA = \O_X$, then we just say that $f$ is an \emph{embedded log
resolution} of $X$ in $M$. An embedded log resolution is said to be
\emph{factorizing} if, furthermore, we have $\I_X\.\O_N = \I_Y\.\O_N(-G)$
where $G$ is a divisor on $N$. 
\end{defi}

The existence of factorizing resolutions is established in \cite{BV03}. The
following adjunction formula holds. For the proof we refer to Lemma~4.4 of
\cite{Eis10}.

\begin{prop}
\label{p:adj}
Let $X$ be a reduced subscheme of pure codimension $e$ of a smooth variety
$M$. Let $g \colon N \to M$ be a factorizing embedded log resolution of $X$
in $M$, and write $\I_X\.\,\O_N = \I_Y \.\,\O_N(-G)$, where $Y$ is the
proper transform of $X$ and $G$ is a divisor in $N$ supported on the
exceptional locus. Then
\[
\Kj YX = \( \Ku NM - e G\)|_Y.
\]
\end{prop}

\subsection{Inversion of adjunction}

The next theorem generalizes the inversion of adjunction formula for
$\Q$-Gorenstein varieties proved in \cite{EM09} and \cite{Kaw08}.

\begin{thm}
\label{t:inv-adj}
Let $X$ be a reduced subscheme of pure dimension $n \ge 2$ and codimension
$e$ of a smooth variety $M$. Then for every proper, effective $\R$-ideal
$\frA$ on $M$ not containing any irreducible component of $X$ is its
vanishing locus, and every  closed subset $T \subset X$, 
we have
\[
\jmld T{X,\frA|_X} = \mld T{M,\frA\.\I_X^e}.
\]
\end{thm}

The theorem has been proven independently by Ishii \cite{Ish11}. The line
of arguments presented here is slightly different (although equivalent at
the core) from that given in \cite{Ish11}, which follows more closely
\cite{EM09}. 

The proof uses the jet schemes $X_m$ and the space of arcs $X_\infty$ of
$X$. We refer the reader to \cite{DL99,EM09} for the basic definitions and
properties of the theory. We will use the description of divisorial
valuations and discrepancies given in \cite{ELM04} in the smooth case and
then extended in \cite{dFEI08} to the singular case. Since here we allow
$X$ to be reducible, the only further extension we need is that of the
notion of quasi-cylinder and codimension. So, we say that a subset $C
\subset X_\infty$ is a \emph{quasi-cylinder} of \emph{codimension} $k$ if
$C$ is contained in the arc space $(X_i)_\infty$ of some irreducible
component $X_i$ of $X$ and is a quasi-cylinder of codimension $k$ in there.
Since $X$ is equidimensional, this implies that if $\ff_m\colon X_\infty
\to X_m$ is the truncation map, then $k = n(m+1) - \dim(\ff_m(C))$ for all
$m \gg 1$, where $n = \dim X$. 

\begin{proof}[Proof of Theorem~\ref{t:inv-adj}]
For every generic point $\e_{X_i}$ of an irreducible component $X_i$ of $X$
we have $\jmld {\e_{X_i}}{X,\frA|_X} = 0 = \mld
{\e_{X_i}}{M,\frA\.\I_X^e}$. We can therefore reduce to the case in which
$T$ has codimension $\ge 1$ in $X$.

To prove one inequality, we take a factorizing embedded log resolution $g
\colon N \to M$ of $(X,\frA|_X)$ in $(M,\frA)$, and let $f \colon Y \to X$
be the induced log resolution. Since every $g$-exceptional divisor that
meets $Y$ intersects it transversely, if $F$ is any such divisor and $E$ is
the divisor cut out by $F$ on $Y$, then $\ord_E(\frA|_X) = \ord_F(\frA)$.
Using Proposition~\ref{p:adj}, it follows by direct comparison along the
$g$-exceptional divisors meeting $Y$ and having center inside $T$ that
either $(X,\frA|_X)$ is not log J-canonical, in which case
$(M,\frA\.\I_X^e)$ is not log canonical and both minimal log
(J-)discrepancies are $-\infty$, or the first pair is log J-canonical and
\[
\jmld T{X,\frA|_X} \ge \mld T{M,\frA\.\I_X^e}.
\]
Here we are using the fact that if $(X,\frA|_X)$ is J-canonical then its
minimal log J-discrepancy is computed on any given log resolution of
$(X,\frA|_X)$, see Proposition~\ref{p:Jmld}.

The reverse inequality follows from the following claim.

\begin{claim}
\label{cl:inv-adj}
For every prime divisor $F$ over $M$ with center inside $X$, there is a
prime divisor $E$ over $X$ with center contained in the center of $F$, and
an integer $q \ge 1$, such that
\begin{equation}
\label{eq:inv-adj}
q\.\aj E{X,\frA|_X} \le \au F{M,\frA\.\I_X^e}.
\end{equation}
\end{claim}

To see that this implies the inequality $\jmld T{X,\frA|_X} \le \mld
T{M,\frA\.\I_X^e}$, note that if $\mld T{M,\frA\.\I_X^e} \ge 0$ then we get
the statement by dividing by $q$ in~\eqref{eq:inv-adj}, whereas if $\mld
T{M,\frA\.\I_X^e} < 0$ then the formula only implies that $\jmld
T{X,\frA|_X} < 0$, but then one deduces immediately that $\jmld
T{X,\frA|_X} = -\infty$ by Proposition~\ref{p:Jmld}.

It remains to prove Claim~\ref{cl:inv-adj}. We consider the \emph{maximal
divisorial set}
\[
W = W^1(F) \subset M_\infty,
\]
where $M_\infty$ is the space of arcs of $M$. By definition, $W$ is the
closure in $M_\infty$ of the set of arcs of $N$ having order of contact one
with $F$ (cf.\ \cite{ELM04}). The set $W$ is an irreducible cylinder in
$M_\infty$. By the results of \cite{ELM04}, the valuation $\val_W$
determined by the vanishing order along the generic arc in $W$ agrees with
the divisorial valuation $\ord_F$, and moreover
\[
\codim(W,M_\infty) = \ku FM + 1.
\]

The intersection $W \cap X_\infty \subset X_\infty$ is a cylinder in
$X_\infty$ and is not contained in the arc space of the singular locus of
$X$, by Lemma~8.3 of \cite{EM09}. Let $C$ be an irreducible component of $W
\cap X_\infty$ that is not contained in the arc space of the singular locus
of $X$. Then $C$ is a quasi-cylinder in $X_\infty$. By Propositions~3.10
and 2.12 of \cite{dFEI08}, we can find a prime divisor $E$ over $X$, and an
integer $q \ge 1$, such that the valuation $\val_C$ associated to the
quasi-cylinder $C$ is equal to the divisorial valuation $q\ord_E$, and
moreover the maximal divisorial set $W^q(E) \subset X_\infty$ is a
quasi-cylinder of codimension
\[
\codim(W^q(E),X_\infty) = q\.(\km EX + 1),
\]
where $\km EX$ is the Mather discrepancy of $E$ over $X$.

We have the following chain of inequalities:
\begin{align*}
q\.(\km EX + 1) &= \codim(W^q(E),X_\infty) \\
& \le \codim(C,X_\infty) \\
& \le \codim(W,M_\infty) + q\.\ord_E(\Jac_X) - e\ord_F(\I_X).
\end{align*}
The first inequality follows from the fact that, as it is explained in the
proof of Propositions~3.10 of \cite{dFEI08}, $C$ is contained in $W^q(E)$,
and the second one by applying Lemma~8.4 in \cite{EM09} as in the proof of
Theorem~8.1 of \cite{EM09}.

Observe that, for any proper ideal $\frb \subset \O_M$ not vanishing on any
component of $X$, we have $\val_C(\frb|_X) \ge \val_W(\frb)$ by the
inclusion $C \subset W$ and the fact that if $\g \in C$ then
$\ord_\g(\frb|_X) = \ord_\g(\frb)$. In particular, this implies that
\[
q\.\ord_E(\frA|_X) \ge \ord_F(\frA)
\]
since $\frA$ is effective. Since $\kj EX = \km EX - \ord_E(\Jac_X)$, we deduce
that
\[
q\.(\kj EX + 1) \le \ku FM + 1 - e\ord_F(\I_X).
\]
Claim~\ref{cl:inv-adj} follows by combining these inequalities.
\end{proof}

\begin{rmk}
\label{r:inv-adj}
With the same notation as in Theorem~\ref{t:inv-adj}, suppose that $X$ has
dimension one. In this case most of the arguments of the proof goes
through, the only problem being that it is no longer true in general that
$\jmld T{X,\frA|_X}$ is $-\infty$ whenever it is negative, and one
concludes in this case that either $\jmld T{X,\frA|_X} = \mld
T{M,\frA\.\I_X^e} \ge 0$ or $0 > \jmld T{X,\frA|_X} \ge \mld
T{M,\frA\.\I_X^e}$, and the latter is $-\infty$ if $e \ge 1$. These
inequalities will suffice, however, in the applications of inversion of
adjunction to the proofs of Corollary~\ref{c:jet-sing} and
Theorem~\ref{t:rat-DB-inv-adj}.
\end{rmk}

\subsection{ACC and semi-continuity}

Inversion of adjunction has several implications regarding the properties
of the invariants of singularities related to J-discrepancies, which
generalize analogous properties known for normal varieties with locally
complete intersection singularities. 

The first application gives the ACC property for the sets of log
J-canonical thresholds in any fixed dimension. The problem arises from a
conjecture of Shokurov for log canonical thresholds in the usual setting
\cite{Sho92}, which has recently being solved for certain classes of
singularities in \cite{dFEM10,dFEM11}. In the framework considered in this
paper we obtain the following unconditional result.

\begin{cor}
\label{c:ACC}
For every integer $n$, the set of log J-canonical thresholds in dimension
$n$
\[
\{\;
\jlct{\fra} 
\mid
\text{$\fra \subset \O_X$, $X$ log J-canonical of pure dimension $n$}
\;\}
\]
satisfies the ascending chain condition. That is, every increasing sequence
in the set is eventually constant.
\end{cor}

The proof of this property is a straightforward extensions of the
corresponding proof given in \cite{dFEM10} in the case of normal varieties
with locally complete intersection singularities. In short, it goes as
follows. The same argument of the proof of Proposition~6.3 of \cite{dFEM10}
shows that if $X$ is a reduced equidimensional scheme with log J-canonical
singularities then $\dim T_pX \le 2\dim X$ for every $p \in X$. Then, using
this bound on the embedded dimension in combination with inversion of
adjunction (Theorem~\ref{t:inv-adj}) one deduces the above ACC property
directly from the analogous property of \emph{mixed log canonical
thresholds} on smooth varieties, which is established in Theorem~6.1 of
\cite{dFEM10}.

The second application of inversion of adjunction regards the
semi-continuity of minimal log J-canonical discrepancies and the
characterization of regular points in terms of these invariants. Once more,
the question originates from a conjecture of Shokurov in the usual setting
of minimal log discrepancies later made more precise by Ambro; for a
discussion of this we refer to \cite{Amb99} and the references therein.
Again, we have unconditional results in our setting. The first statement
appears also in \cite{Ish11}.

\begin{cor}
\label{c:semi-contII}
For every effective pair $(X,\frA)$ where $X$ is a reduced equidimensional
scheme, the function on closed points
\[
\jmld {-}{X,\frA} \colon X \to \{-\infty\} \cup \R_{\ge 0}, 
\quad p \mapsto \jmld p{X,\frA},
\]
is lower semi-continuous in the Zariski topology. 
\end{cor}

\begin{cor}
\label{c:semi-contI}
Let $X$ be a reduced equidimensional scheme. For every Grothendieck point
$\e\in X$ we have
\[
\jmld \e X \le \codim(\e,X),
\]
and equality holds if and only if $X$ is smooth at $\e$.
\end{cor}

\begin{rmk}
Since $\jmld \e X\in \Z \cup \{-\infty\}$, it follows that $X$ is regular
at $\e$ if and only if $\jmld \e X > \codim(\e,X) - 1$.
\end{rmk}

Once inversion of adjunction is in place, the proofs of these results are
standard. The proof of Corollary~\ref{c:semi-contII} follows step by step
the arguments of the corresponding result in \cite{EM04}. Regarding
Corollary~\ref{c:semi-contI} one uses induction on the embedded codimension
of $X$ at $\e$, as explained for instance in Remark~4.2 of \cite{dFE10} for
the locally complete intersection case. 

\begin{rmk}
On a completely different topic, one also sees from inversion of adjunction
that the bound on Castelnuovo--Mumford regularity proven in Corollary~1.4
of \cite{dFE10} holds for every reduced equidimensional projective scheme
$V \subset \P^n$ with log J-canonical singularities.
\end{rmk}

\section{Jet schemes}
\label{s:jet}

It is known that Mather discrepancies can be computed using the codimension
of certain sets in the arc space. We give now an analogous description for
Jacobian discrepancies which involves the tangent space to the arc space.
Throughout the section, let $X$ be a reduced scheme of pure dimension $n$.

\subsection{Mather discrepancies}

Recall that to a divisor $E \subset Y$ over $X$ one associates its
\emph{maximal divisorial set} $W^1(E)$, namely, the closure in $X_\infty$
of the image of the set of arcs on $Y$ having order of contact one with
$E$. It is an irreducible quasi-cylinder in the arc space $X_\infty$, and
its codimension measures the Mather log discrepancy
\[
\km EX + 1 = \codim\(W^1(E), X_\infty\),
\]
see \cite{dFEI08} for details. More precisely, if $\ff_m \colon X_\infty
\to X_m$ is the truncation map, then it follows by Lemma~3.4 of \cite{DL99}
(cf.~Lemma~3.8 and the argument in the proof of Theorem~3.9 of
\cite{dFEI08}) that $W^1(E)$ is the closure of a quasi-cylinder over a
constructible subset of $X_{m_0}$ where $m_0 = 2\ord_E(\Jac_X)$, and thus
we have
\[
\km EX + 1 = n(m+1) - \dim\(\ff_m(W^1(E))\) \fall m \ge 2\ord_E(\Jac_X).
\]
These formulas are the geometric manifestation of the change-of-variables
formula in motivic integration \cite{DL99}.

\subsection{Jacobian discrepancies}

To obtain an analogous description for Jacobian discrepancies, we need to
consider the total spaces of the tangent sheaves of the arc space and of
the jet schemes. We denote them by $TX_\infty$ and $TX_m$. These spaces are
schemes, and their functors of points are easy to describe; note however
that $TX_\infty$ is not of finite type. For example, when $X$ is affine the
$\C$-valued points are given by
\begin{align*}
TX_\infty(\C) &=
\Hom \( \Spec \C[\![t]\!][\ep]/(\ep^2),\,  X \), \\
TX_m(\C) &=
\Hom \( \Spec \C[t, \ep]/(t^{m+1}, \ep^2),\,  X \).
\end{align*}
We have natural projections
\[
\p_\infty     \colon TX_\infty     \longrightarrow  X_\infty,
\qquad
\p_m          \colon TX_m          \longrightarrow  X_m
\]
from the tangent spaces to their bases. Given an arc $\a \in X_\infty$ and
a \emph{liftable} jet $\b \in \ff_m(X_\infty) \subset X_m$, we are
interested in the restrictions of the tangent spaces over these points,
which we denote by
\[
TX_\infty|_{\a} := \p_\infty^{-1} \( \a \),
\qquad
TX_m|_{\b} := \p_m^{-1} \( \b \).
\]

\begin{prop}
\label{p:jet-discr}
Let $X$ be a reduced scheme of pure dimension $n$. Consider a liftable jet
$\b \in \ff_m(X_\infty)$, and let $L = \C(\b)$ be the field of definition
of $\b$. Then
\[
\dim_L\( TX_m|_{\b} \) = 
n(m+1) + \ord_{\b} \( \Jac_X \) \fall m \geq \ord_{\b}(\Jac_X).
\]
\end{prop}

\begin{proof}
Let $\a \in X_\infty$ be a lift of $\b$. We can assume $X$ is affine, and
embedded in $M = \A^{n+e}$. We denote by $x_1, \dots, x_{n+e}$ the
coordinates in $M$, and let $I_X = (f_1, \dots, f_r)$ be the equations of
$X$. Then the arc $\a$ is given by a vector $(\a_1, \dots, \a_{n+e})$ with
entries $\a_i = \a_i(t)$ in the power series ring $L[\![t]\!]$. Since $\a$
is in $X_\infty$, we know that $f_j(\a) = 0$ for all $j$.

The restriction $TM_\infty|_\a$ can be thought as a free module over
$L[\![t]\!]$ of rank $n+e$. Specifically, every element $\x \in
TM_\infty|_\a$ can be written in the form
\[
\x = \a + v \, \ep,
\]
where $v = (v_1, \dots, v_{n+e})$ is a vector with coefficients in
$L[\![t]\!]$ and $\ep$ is a fixed variable verifying $\ep^2 = 0$. The
tangent vector $\x$ belongs to $TX_\infty|_\a$ if $f_j(\x) = 0$ for all
$j$. Using the Taylor expansion, we get
\[
f_j(\x) = f_j(\a) + \sum_{i=1}^{n+e} 
\frac {\partial f_j}{\partial x_i} (\a) \, v_i \, \ep.
\]
Let $J = \( \frac {\partial f_j} {\partial x_i} \)$ be the Jacobian matrix.
Since $f_j(\a) = 0$, we see that
\[
\x \in TX_\infty|_\alpha
\quad\Longleftrightarrow\quad
J(\a) \, v = 0.
\]
In other words, $TX_\infty|_\a$ can be computed inside of $TM_\infty|_\a$
as the kernel of $J(\a)$.

Analogous statements hold for the jet $\b$. In this case, $TM_m|_\b$ is a
free module over $L[t]/(t^{m+1})$, also of rank $n+e$, and $TX_m|_\b$ is
the kernel of $J(\b)$:
\[
\b + w \, \ep \in TX_m|_\b
\quad\Longleftrightarrow\quad
J(\b) \, w = 0 \pmod{t^{m+1}}.
\]

The goal is therefore to compute the dimension of the kernel of $J(\b)$,
and this can be done easily by diagonalizing the matrix. In our case, we
can diagonalize both $J(\b)$ and $J(\a)$ simultaneously. More precisely,
the structure theorem for finitely generated modules over PID's tells us
that we can find invertible matrices $A$ and $B$ with coefficients in
$L[\![t]\!]$ such that
\[
A \. J(\a) \. B = D
\]
where $D$ is a diagonal matrix. Notice that $D$ is not a square matrix: by
\emph{diagonal} we mean that all of its entries $d_{ij}$ are zero except
when $i=j$. We can also assume that $D$ is of the form
\[
D = \operatorname{diag}(t^{a_1}, \dots, t^{a_s}, 0, \dots, 0)
\]
with $0 \leq a_1 \leq a_2 \leq \dots \leq a_s$.

Notice that $J(\b)$ is the truncation of $J(\a)$ to order $m$, so if we
denote by $A_m$, $B_m$, and $D_m$ the truncations of $A$, $B$, and $D$, we
have
\[
A_m \. J(\b) \. B_m = D_m.
\]
The matrices $A_m$ and $B_m$ are also invertible, so in particular $J(\b)$
and $D_m$ have isomorphic kernels. The matrix $D_m$ is given by
\[
D_m = \operatorname{diag}(t^{a_1}, \dots, t^{a_l}, 0, \dots, 0),
\]
where $l \leq s$ is picked so that $a_k > m$ for $k > l$. Its kernel is
\[
(t^{m+1-a_1})/(t^{m+1})
\oplus \dots \oplus
(t^{m+1-a_l})/(t^{m+1})
\oplus \( L[t]/(t^{m+1}) \)^{n+e-l},
\]
which has dimension $a_1 + \dots + a_l + (n+e-l)(m+1)$ over $L$.

Recall that the matrix $J$ gives a presentation for the module of
differentials $\Om_X$. Therefore, the $k$-th Fitting ideal $\Fitt^k(\Om_X)$
is generated by the minors of $J$ of size $n+e-k$. In particular, the
orders of vanishing
\[
\ord_\a\(\Fitt^k(\Om_X)\),
\qquad\quad
\ord_\b\(\Fitt^k(\Om_X)\)
\]
are the smallest order of vanishing of a minor of size $n+e-k$ of $J(\a)$
and $J(\b)$. Since $A$, $A_m$, $B$, and $B_m$ are invertible, these orders
can be computed using $D$ and $D_m$, and we see that
\[
\ord_\a\(\Fitt^k(\Om_X)\) = 
\begin{cases}
a_1 + \dots + a_{n+e-k}  & \text{if ${n+e-k}\leq s$,} \\
\infty  & \text{if ${n+e-k} > s$,}
\end{cases}
\]
and
\[
\ord_\b\(\Fitt^k(\Om_X)\) = 
\begin{cases}
\min(a_1 + \dots + a_{n+e-k},\, m+1)  & \text{if ${n+e-k}\leq l$,} \\
m+1  & \text{if ${n+e-k} > l$.}
\end{cases}
\]

Since $X$ has pure dimension $n$, we know that $\Fitt^{n-1}(\Om_X) = 0$, so
$\a$ vanishes along it with order $\infty$, and we get that $e+1 > s$.
Recall that $\Fitt^n(\Om_X) = \Jac_X$. Using the hypothesis that
$\ord_\b(\Jac_X) \leq m$, we get that $e \leq l$. Therefore $e = l = s$,
and
\[
\ord_\b(\Jac_X) = a_1 + \dots + a_e.
\]
Finally, this implies that the kernel of $D_m$ has dimension
$\ord_\b(\Jac_X) + n (m+1)$ over $L$, as required.
\end{proof}

\begin{rmk}
It is essential in Proposition~\ref{p:jet-discr} that the jet $\beta$ is
liftable (that is, in the image of $X_\infty$). In the proof, one cannot
use the fact that $\Fitt^{n-1}(\Om_X) = 0$ to show directly that $e+1 > l$.
This is due to the presence of the $\min$ in the formula for the order of
$\beta$.
\end{rmk}

Given a prime divisor $E$ over $X$, we denote by $\e_{E,m}$ the generic
point of the truncation $\ff_m(W^1(E))$ of the maximal divisorial set
$W^1(E) \subset X_\infty$, and consider the restriction $TX_m|_{\e_{E,m}}$
of the tangent space of $X_m$ over this point. 

\begin{thm}
\label{t:jet-discr}
With the above notation, for every $E$ we have
\[
\kj EX + 1 =
2n(m+1) - \dim_\C \big(TX_m|_{\e_{E,m}}\big)
\fall m \ge 2\ord_E(\Jac_X).
\]
\end{thm}

\begin{proof}
For $m \ge \ord_E(\Jac_X)$, the order of $E$ along $\Jac_X$ is computed by
the jet $\e_{E,m}$, and the previous proposition gives
\[
-\ord_E(\Jac_X) = n(m+1) - \dim_{\C(\e_{E,m})} \( TX_m|_{\e_{E,m}} \).
\]
On the other hand, the description of Mather discrepancies in terms of the
codimension of $W(E)$ gives
\[
\km EX + 1 = n(m+1) - \dim_\C \(\e_{E,m}\)
\]
for all $m \ge 2\ord_E(\Jac_X)$. The assertion follows then by combining
the two formulas.
\end{proof}

Since $2n(m+1)$ is the dimension of the irreducible component of $TX_m$
dominating $X$, one can consider the right-hand side in the formula in the
theorem as a `virtual codimension' of $TX_m|_{\e_{E,m}}$, although not with
respect to the full $TX_m$ but rather in relation to this distinguished
component.

It would be interesting to find a way to read the condition characterizing
Jacobian discrepancies in Theorem~\ref{t:jet-discr} all the way up, on the
tangent space of $X_\infty$, similarly as for Mather discrepancies which
are detected by the codimension of certain quasi-cylinders. Unfortunately
this seems difficult, as in general the map $TX_\infty|_{\e_E} \to
TX_m|_{\e_{E,m}}$ is not dominant.

\subsection{Jet interpretation of singularities}

One more application of inversion of adjunction regards the
characterization of J-canonical and log J-canonical singularities in terms
of the dimensions of the jet schemes. In a similar fashion,
Theorem~\ref{t:jet-discr} yields a characterization in terms of the
dimension of the tangent spaces of the jet schemes. The next result extends
and generalizes the analogous properties established for locally complete
intersection varieties in \cite{Mus01,EM04}. 

\begin{cor}
\label{c:jet-sing}
Let $X$ be a reduced scheme of pure dimension $n$. For any prime divisor
$E$ over $X$ and any $m$, we denote by $\e_{E,m}$ the image in $X_m$ of the
generic point of the maximal divisorial set $W^1(E) \subset X_\infty$. Then
the following are equivalent:
\begin{enumerate}
\item
$X$ is log J-canonical. 
\item
$\dim X_m = n(m+1)$ for every $m$. 
\item
$\dim TX_m|_{\e_{E,m}} \le 2n(m+1)$ for all $E$ and any $m \ge 2\ord_E(\Jac_X)$. 
\end{enumerate}
Similarly, the following are equivalent:
\begin{enumerate}
\item[(a')]
$X$ is J-canonical.
\item[(b')]
$\dim X_m = n(m+1)$ for every $m$, and every irreducible component of $X_m$
of maximal dimension dominates an irreducible component of $X$. 
\item[(c')]
$\dim TX_m|_{\e_{E,m}} < 2n(m+1)$ for all $E$ and any $m \ge 2\ord_E(\Jac_X)$.
\end{enumerate}
Moreover, in (b) and (b') is enough to check the condition for all $m$ such
that $m+1$ is sufficiently divisible, and in (c) and (c') it is enough to
check the condition just for the prime divisors $E_1,\dots,E_k$ appearing
in a log resolution of $X$.
\end{cor}

\begin{proof}
The equivalences (a)~$\Leftrightarrow$~(b) and (a')~$\Leftrightarrow$~(b')
come from inversion of adjunction. The argument is quite standard. Let $S
\subset X$ denote the singular locus of $X$. By definition $X$ is log
J-canonical if and only if $\jmld SX \ge 0$ (resp., $\jmld SX \ge 1$). If
$X$ is embedded in a smooth variety $M$, then by Theorem~\ref{t:inv-adj}
(see Remark~\ref{r:inv-adj} if $\dim X = 1$) this is equivalent to $\mld
S{M,eX} \ge 0$ (resp., $\mld S{M,eX} \ge 1$), where $e = \codim(X,M)$.
Therefore the equivalences follow from the straightforward generalization
of Theorems~3.1 and~3.2 of \cite{Mus01} to reduced equidimensional schemes.
Note that these theorems also imply that it is enough to check the
conditions in (b) and (b') when $m+1$ is sufficiently divisible.

The equivalences (a)~$\Leftrightarrow$~(c) and (a')~$\Leftrightarrow$~(c')
and the last assertion are a straightforward consequence of the definitions
of singularities, Theorem~\ref{t:jet-discr}, and Proposition~\ref{p:Jmld}.
\end{proof}

\begin{rmk}
While the original proofs in \cite{Mus01} make explicit use of motivic
integration, it is now well understood by the experts that the results only
need the underlying geometric properties of the jet schemes, and one
obtains quicker proofs using for instance the point of view of maximal
divisorial sets developed in \cite{ELM04}.
\end{rmk}

\section{Multiplier ideals}
\label{s:mult-id}

In this section we introduce multiplier ideals in our framework, and use
them to measure the gap between the dualizing sheaf of a normal variety and
its Grauert--Riemenschneider canonical sheaf. We refer to \cite{Laz04} for
an introduction to multiplier ideals in the usual setting.

\subsection{Mather and Jacobian multiplier ideals}

In the following, let $X$ be a normal variety. Consider a proper $\R$-ideal
$\frA = \prod_k \fra_k^{c_k}$ on $X$, and let $f \colon Y \to X$ be a log
resolution of the pair $(X,\frA)$. For short, we denote by $Z(\frA\.\O_Y)
:= \sum_k c_k\.Z(\fra_k\.\O_Y)$ the divisor determined by $\frA$ on $Y$.

\begin{defi}
The \emph{Mather multiplier ideal} of $(X,\frA)$ is the coherent sheaf of
fractional ideals
\[
\MIm{\frA} := f_*\O_Y(\Km YX - \rd{Z(\frA\.\O_Y)}).
\]
and the \emph{Jacobian multiplier ideal sheaf} of $(X,\frA)$ is the
coherent sheaf of fractional ideals
\[
\MIj{\frA} := f_*\O_Y(\Kj YX - \rd{Z(\frA\.\O_Y)}).
\]
\end{defi}

A standard argument using Proposition~\ref{p:K-compos} shows that the
definition of Mather and Jacobian multiplier ideals is independent of the
particular log resolution. The proof follows the exact same lines of the
proof of the analogous statement for multiplier ideals on smooth varieties,
see Theorem~9.2.18 of \cite{Laz04}. If $X$ is locally complete
intersection, then by Corollary~\ref{c:lci-K_Y/X} we have $\MIj{\frA} =
\MIu{\frA}$, the usual multiplier ideal. If $X$ is smooth, then we also
have $\MIm{\frA} = \MIu{\frA}$. 

\begin{rmk}
\label{r:N-J-mult-id}
Since clearly $\MIj{\frA} = \MIm{\frA\.\Jac_X}$, the two theories of
multiplier ideals are equivalent as long as one allows non-effective pairs.
If one restricts the setting to effective pairs, then Jacobian multiplier
ideals can be regarded as a special case of Mather multiplier ideals.
\end{rmk}

\begin{prop}
\label{p:eff-ideal}
On a normal variety, both Mather and Jacobian multiplier ideals define an
ideal sheaf (as opposed to a fractional ideal sheaf) if the pair is
effective in codimension one.
\end{prop}

\begin{proof}
By the previous remark, it is enough to check this property for Mather
multiplier ideals. Given a log resolution $f \colon Y \to X$ of a pair
$(X,\frA)$, write
\[
\Km YX - \rd{Z(\frA\.\O_Y)} = P - N
\]
where $P$ and $N$ are effective divisors with no common components, and
consider the exact sequence
\[
0 \to \O_Y(-N) \to \O_Y(P-N) \to \O_P(P-N) \to 0.
\]
If the pair is effective in codimension one, then $P$ is an exceptional
divisor and $f_*\O_P(P-N) \subset f_*\O_P(P) = 0$ by a well-known lemma of
Fujita (see Lemma~1-3-2 of \cite{KMM87}). This implies that 
\[
\MIm{\frA} = f_*\O_Y(P-N) = f_*\O_Y(-N) \subset \O_X. \qedhere
\]
\end{proof}

Mather and Jacobian multiplier ideals satisfy similar properties as the
usual multiplier ideals on smooth varieties. We list here a few, leaving to
the reader the details of the proofs. 

\begin{prop}
\label{p:properties}
Let $X$ be a normal variety.
\begin{enumerate} 
\item
If $\ov\frA = \prod_k \ov\fra_k^{c_k}$, then $\MIm{\ov\frA} = \MIm{\frA}$
and $\MIj{\ov\frA} = \MIj{\frA}$.
\item
If $\frB = \prod_k \frb_k^{d_k}$ with $\frb_k \subset \fra_k$ and $d_k \ge
c_k$ for all $k$, then $\MIm{\frB} \subset \MIm{\frA}$ and $\MIj{\frB}
\subset \MIj{\frA}$.
\item
$\fra \subset \MIm{\fra}$ for every ideal sheaf $\fra \subset \O_X$. In
particular, $\MIm{\O_X} = \O_X$.
\item
$\fra\.\MIj{\O_X} \subset \MIj{\fra}$ for every ideal sheaf $\fra \subset \O_X$. 
\end{enumerate}
\end{prop}

The proof of the next proposition is essentially the same as that of the
transformation rule for the usual multiplier ideals, see Proposition~9.2.32
of \cite{Laz04}. We outline it for the convenience of the reader. A similar
property holds for Jacobian multiplier ideals: the same proof goes through,
or equivalently one can deduce it from the case treated below using
Remark~\ref{r:N-J-mult-id}.

\begin{prop}
\label{p:bir-transf}
Suppose that the $\R$-ideal $\frA = \prod_i \fra_i^{c_i}$ has integral
exponents $c_i \in \Z$. Let $f \colon Y \to X$ be a resolution of $X$
factoring through the blow-up of the Jacobian ideal $\Jac_X$, and write
$\fra_i\.\O_Y = \frb_i\.\O_Y(-D_i)$ where $\frb_i \subset \O_Y$ and $D_i$
is a divisor. Let $\frB = \prod_i \frb_i^{c_i}$ and $D = \sum_i c_i D_i$.
Then
\[
\MIm{\frA} = f_*\big(\MIu{\frB}\otimes\O_Y(\Km YX - D)\big).
\]
\end{prop}

\begin{proof}
Let $f' \colon Y' \to X$ be a log resolution of $(X,\frA)$ factoring
through $f$ and a morphism $g \colon Y' \to Y$. Write $\fra_i\.\O_{Y'} =
\O_{Y'}(-A_i)$ and $\frb_i\.\O_{Y'} = \O_{Y'}(-B_i)$, and let $A =
\sum_ic_iA_i$ and $B = \sum_ic_iB_i$. Note that $A = B + g^*D$. By
definition, we have $\MIu{\frB} = g_*\O_{Y'}(\Ku {Y'}Y - B)$, and therefore
we get
\begin{align*}
\MIm{\frA} &= f'_*\O_{Y'}(\Km {Y'}X - A) \\
&= f'_*\big(\O_{Y'}(\Ku {Y'}Y - B) \otimes g^*\O_Y(\Km YX - D)\big) \\
&= f_*\big(\MIu{\frB} \otimes \O_Y(\Km YX - D)\big)
\end{align*}
by projection formula and Proposition~\ref{p:K-compos}.
\end{proof}

The following characterizations of singularities come directly from the
definitions.

\begin{prop}
Let $X$ be a normal variety.
\begin{enumerate}
\item
$X$ is J-canonical if and only if $\MIj{\O_X} = \O_X$.
\item
$X$ is log J-canonical if and only if $\MIj{\Jac_X^{-\l}} = \O_X$ for all
$\l > 0$. 
\end{enumerate}
\end{prop}

\begin{rmk}
\label{r:EIM}
The same notion of multiplier ideals has been 
independently introduced and studied in \cite{EIM}.
\end{rmk}

\subsection{Grauert--Riemenschneider canonical sheaf of a variety}

Next we discuss how the canonical sheaf $\om_X$ of a normal variety $X$
relates to the \emph{Grauert--Riemenschneider canonical sheaf} of $X$,
which, we recall, is defined by
\[
\grom X := f_*\om_Y
\] 
for any resolution $f \colon Y \to X$, see \cite{GR70}.  There is a natural
inclusion $\grom X \subset \om_X$ and our goal is to give a measure of the
gap between the two sheaves.

When $X$ is locally complete intersection (or, more generally, when $\om_X$
is invertible) we have 
\[
\grom X \cong \om_X \otimes \MIu{\O_X}.
\]
Indeed, by definition $\MIu{\O_X} = f_*\O_Y(\Ku YX) \cong
f_*\om_Y\otimes\om_X^{-1}$ where $f\colon Y \to X$ is any log resolution of
$X$. We can rewrite this formula using the Mather multiplier ideal of the
Jacobian. The observation is that, if $X$ is locally complete intersection
and $Y \to X$ is a log resolution, then $\Ku YX = \Kj YX$ and thus
$\MIu{\O_X} = \MIj{\O_X} = \MIm{\Jac_X}$, see Corollary~\ref{c:lci-K_Y/X}
and Remark~\ref{r:N-J-mult-id}. Therefore the formula can be written as
\[
\grom X \cong \om_X \otimes \MIm{\Jac_X}.
\]

\begin{rmk}
Since $\mom X \cong \om_X \otimes \Jac_X$ when $X$ is locally complete
intersection, there is in this case a correspondence between the chain of
inclusions $\mom X \subset \grom X \subset \om_X$ and the inclusions
$\Jac_X \subset \MIm{\Jac_X} \subset \O_X$, determined by tensoring by
$\om_X^{-1}$.
\end{rmk}

In general, when $X$ is not locally complete intersection, we pick a
reduced, locally complete intersection scheme $V$ containing $X$, of the
same dimension. If $f \colon Y \to X$ is a log resolution of $(X,
\Jac_V|_X)$ and $\^f\colon Y \to \^X$ is the induced map on the Nash
blow-up of $X$, then we have $\Jac_V|_X\.\O_Y \cong \O_Y(\^f^*\^K_X)
\otimes f^*(\om_V^{-1}|_X)$ by Proposition~\ref{p:taut-bundles}. This gives
$\MIm{\Jac_V|_X} \cong f_*\om_Y \otimes \om_V^{-1}|_X$, and therefore 
\begin{equation}
\label{eq:om-lci}
\grom X \cong \om_V|_X \otimes \MIm{\Jac_V|_X}.
\end{equation}

\begin{rmk}
For any $V$ as above, there is a correspondence between the inclusions
$\mom X \subset \grom X \subset \om_V|_X$ and the inclusions $\Jac_V|_X
\subset \MIm{\Jac_V|_X} \subset \O_X$, given by tensoring by
$\om_V^{-1}|_X$.
\end{rmk}

The idea now is to assemble the isomorphisms given in \eqref{eq:om-lci} by
letting $V$ vary. The next theorem is the main result of this section. The
key ingredient is the \emph{lci-defect ideal} $\lcid X$ of $X$, which was
defined in Section~\ref{s:Nash} by
\[
\lcid X := \sum_{V} \lcid{X,V}
\]
where $\lcid{X,V}$ is the ideal determined by the image of $\om_X \to
\om_V|_X$. Note that since $X$ is normal, $\frd_X$ is trivial in
codimension one and thus $\MIj{\frd_X^{-1}}$ is an ideal sheaf by
Proposition~\ref{p:eff-ideal}. 

\begin{thm}
\label{t:can}
For every normal variety $X$, we have
\[
\big( \grom X : \om_X \big) = \MIj{\lcid X^{-1}}.
\]
\end{thm}

\begin{rmk}
Using Lemma~\ref{l:colon} below and Remark~\ref{r:N-J-mult-id}, the formula
in the theorem can be rewritten in the following equivalent ways:
\[
\big( \grom X : \om_X \big) 
= \big(\MIj{\O_X} : \lcid X \big) 
= \big(\MIm{\Jac_X} : \lcid X \big).
\]
\end{rmk}

\begin{proof}[Proof of Theorem~\ref{t:can}]
Since the question is local, we can assume that $X$ is affine. We fix an
embedding of $X$ in a smooth affine variety $M$, and denote $e =
\codim(X,M)$. 

Let $T$ be an irreducible algebraic family parametrizing reduced, complete
intersections $V \subset M$ of codimension $e$ containing $X$. The family
is constructed as an open set of the Grassmannian of $e$-tuples of linear
combinations among a fixed set of generators of the ideal $\I_X$ of $X$ in
$M$. We have $\om_X = \lcid{X,V}\otimes\om_V|_X$ by definition and $\grom X
= \MIm{\Jac_V|_X} \otimes \om_V|_X$ by \eqref{eq:om-lci}, and hence
\[
\big( \grom X : \om_X \big) = 
\big( \MIm{\Jac_V|_X}\otimes \om_V|_ X : \lcid{X,V}\otimes\om_V|_X \big) = 
\big( \MIm{\Jac_V|_X} : \lcid{X,V}\big) 
\]
since $\om_V|_X$ is invertible. Using Lemma~\ref{l:colon} below, we get
\[
\big( \grom X : \om_X \big) = 
\MIm{\Jac_V|_X\.\lcid{X,V}^{-1}}. 
\]
Therefore, in order to prove the theorem we are reduced to prove the
following identity:
\begin{equation}
\label{eq:identity}
\MIm{\Jac_V|_X\.\lcid{X,V}^{-1}} = \MIm{\Jac_X\.\lcid X^{-1}}. 
\end{equation}
Since the left hand side does not depend on $V$, we can assume that $V$ is
general in $T$. 

Let $f \colon Y \to X$ be a log resolution of $(X,\Jac_X\.\lcid X^{-1})$,
and write
\[
\Jac_X\.\O_Y = \O_Y(-A), \qquad \lcid X\.\O_Y = \O_Y(-B).
\]
We have $\Jac_X = \sum_{V \in T} \Jac_V|_X$, and $\lcid X$ has the same
integral closure of $\lcid{X/M} = \sum_{V \in T} \lcid{X,V}$ (see
Proposition~\ref{p:indep-M}). Therefore, if $V$ is sufficiently general
then by Lemma~\ref{l:v(a_t)} below we have
\[
\Jac_V|_X\.\O_Y = \fra_V\.\O_Y(-A), 
\qquad \lcid {X,V}\.\O_Y = \frb_V\.\O_Y(-B),
\]
where $\fra_V, \frb_V \subset\O_Y$ do not vanish along any exceptional
divisor. Moreover, we have
\[
\sum_{V \in T} \fra_V = \O_Y.
\]
This follows again from the fact that, since $\sum_{V \in T} \Jac_V|_X =
\Jac_X$, for every divisorial valuation $\n$ we can find a $V$ in $T$ such
that $\n(\Jac_V|_X)= \n(\Jac_X)$. 

By taking $V = X \cup X'$ general, we can assume by Bertini that $X'$
intersects $X$ transversally in codimension one (recall that $X$ is normal,
hence smooth in codimension one). In particular, it follows that $\fra_V$
and $\frb_V$ agree in codimension one. Indeed since these sheaves do not
vanish on exceptional divisors, this becomes a computation on $X$, and it
is easy to see that if $X'$ intersects $X$ transversally in codimension
one, then at the generic point of each irreducible component of $X \cap X'$
both ideals are reduced.

Note on the other hand that $\Jac_V|_X\.\O_Y$ is locally principal, since
$f$ factors through the blow-up of $\Jac_X$ and hence through the Nash
blow-up of $X$, which is isomorphic to the blow-up of $\Jac_V|_X$.
Therefore $\fra_V$ is locally principal, and hence there is an inclusion
$\frb_V \subset \fra_V$ that is an identity in codimension one. This
implies that
\[
\MIu{\fra_V\.\frb_V^{-1}} = \O_Y.
\]
Using Proposition~\ref{p:bir-transf}, we get
\begin{align*}
\MIm{\Jac_V|_X\.\lcid{V,X}^{-1}} 
&= f_*\big(\MIu{\fra_V\.\frb_V^{-1}}\.\O_Y(\Km YX - A + B)\big) \\
&= f_*\O_Y(\Km YX - A + B) \\
&= \MIm{\Jac_X\.\lcid X^{-1}}. 
\end{align*}
This proves the identity \eqref{eq:identity}.
\end{proof}

\begin{rmk}
If in the proof one takes $f$ so that it also factors through the Nash
transformation of $X$ relative to $\om_X$, then $\lcid{X,V}\.\O_Y$ is
locally principal too and hence $\fra_V= \frb_V$. This step is however not
necessary in the proof. 
\end{rmk}

\begin{lem}
\label{l:colon}
On a normal variety $X$, for every two ideals $\fra,\frb \subset \O_X$ we
have
\[
\big( \MIm{\fra} : \frb \big) = 
\MIm{\fra\.\frb^{-1}}. 
\]
\end{lem}

\begin{proof}
Let $f \colon Y \to X$ be a log resolution of $(X,\fra \.\frb)$, and write
$\fra\.\O_Y = \O_Y(-A)$ and $\frb\.\O_Y = \O_Y(-B)$. Then 
\begin{align*}
x \in \big( \MIm{\fra} : \frb \big) 
\quad &\Longleftrightarrow \quad x \.\frb \subset \MIm{\fra} \\
\quad &\Longleftrightarrow \quad f^*x \.\O_Y(-B) \subset \O_Y(\Km YX - A) \\
\quad &\Longleftrightarrow \quad f^*x \in \O_Y(\Km YX - A + B)  \\
\quad &\Longleftrightarrow \quad x \in \MIm{\fra\.\frb^{-1}}. \qedhere
\end{align*}
\end{proof}

\begin{lem}
\label{l:v(a_t)}
On a variety $M$, let $\fra_t \subset \O_M$, $t \in T$, be an algebraic
family of ideal sheaves, and let $\fra = \sum_{t\in T} \fra_t$. Then for
every divisorial valuation $\n$ of $\O_M$ there is a non-empty open set
$T_\n \subset T$ such that $\n(\fra_\n) = \n(\fra)$ for every $\n \in
T_\n$.
\end{lem}

\begin{proof}
For every $\n$ we have $\n(\fra) = \min_{t \in T} \n(\fra_t)$, and hence
the assertion follows from the semi-continuity of the function $\fra_t
\mapsto \n(\fra_t)$. 
\end{proof}

\subsection{Grauert--Riemenschneider canonical sheaf of a pair}

Let $\frA = \prod_k \fra_k^{c_k}$ be a proper $\R$-ideal on a normal
variety $X$. Associated to the pair $(X,\frA)$, we consider the sheaf
\[
\grom {(X,\frA)} := f_*\om_Y(- \rd{Z(\frA\.\O_Y}).
\]
where $f \colon Y \to X$ is any log resolution of $(X,\frA)$. We call
$\grom{(X,\frA)}$ the \emph{Grauert--Riemenschneider canonical sheaf of
the pair}.

It is well-known that the definition of $\grom{(X,\frA)}$ is independent of
the choice of log resolution. If $(X,\frA)$ is effective in codimension
one, then $\grom{(X,\frA)}$ is a subsheaf of $\grom X$ and thus of $\om_X$.
Motivated by an analogous definition in positive characteristics due to
Smith \cite{Smi95}, this sheaf has been considered before with the name of
\emph{multiplier submodule} by several authors, see in particular
\cite{HS03,Bli04,ST08}. We however prefer to view $\grom{(X,\frA)}$ as a
`perturbation' of the Grauert--Riemenschneider canonical sheaf of $X$,
hence the terminology and notation adopted here. 

Theorem~\ref{t:can} generalizes as follows.

\begin{thm}
\label{t:can(X,Z)}
For every proper $\R$-ideal $\frA$ on a normal variety $X$, we have
\[
\big( \grom {(X,\frA)} : \om_X \big) = \MIj{\frA\.\lcid X^{-1}}.
\]
\end{thm}

\begin{proof}
The proof proceeds along the same lines of that of Theorem~\ref{t:can}.
Using 
\[
\grom {(X,\frA)} = \MIm {\frA\.\Jac_V|_X} \otimes \om_V|_X,
\]
which is a straightforward generalization of \eqref{eq:om-lci}, we get this
time
\[
\big( \grom {(X,\frA)} : \om_X \big) = 
\MIm{\frA\.\Jac_V|_X\.\lcid{V,X}^{-1}}. 
\]
It is therefore enough to prove that
\[
\MIm{\frA\.\Jac_V|_X\.\lcid{V,X}^{-1}} = \MIm{\frA\.\Jac_X\.\lcid X^{-1}}. 
\]
This follows by the same arguments leading to \eqref{eq:identity} in the
proof of Theorem~\ref{t:can}.
\end{proof}

\begin{rmk}
When $\om_X$ is invertible, the theorem implies that $\MIj{\frA\.\lcid
X^{-1}}= \MIu{X,\frA}$ for every proper $\R$-ideal $\frA$ on $X$, since in
this case $\grom {(X,\frA)} \cong \om_X \otimes \MIu{X,\frA}$. In
particular, $X$ is canonical (resp., log canonical) if and only if the pair
$(X,\lcid X^{-1})$ is J-canonical (resp., log J-canonical). Both properties
can also be deduced directly from Proposition~\ref{p:Q-Gor} as in this case
$\lcid X = \lcid{1,X}$.
\end{rmk}

\begin{rmk}
Grauert and Riemenschneider proved that the Kodaira vanishing theorem holds
on any normal projective variety $X$ if one considers the sheaf $\grom X$
in place of the canonical sheaf $\om_X$ \cite{GR70}. More generally, a
standard application of the Kawamata--Viehweg Vanishing theorem implies the
following general vanishing property. Let $\frA = \prod_i \fra_i^{c_i}$ be
an effective proper $\R$-ideal on $X$, and suppose that, for every $i$,
$D_i$ is a Cartier divisor on $X$ such that $\O_X(D_i)\otimes\fra_i$ is
globally generated. Then for every Cartier divisor $L$ such that $L - \sum
_i c_iD_i$ is a nef and big $\R$-divisor, we have
\[
H^j\big(\grom{(X,\frA)} \otimes \O_X(L)\big) = 0 \fall j > 0.
\]
In view of Theorem~\ref{t:can(X,Z)}, such vanishing can be interpreted as a
Nadel-type vanishing theorem for the Jacobian multiplier ideal
$\MIj{\frA\.\lcid X^{-1}}$.
\end{rmk}

\section{Rational and Du Bois singularities}
\label{s:rat-DB}

As an application of the main results of the paper, we give in this section
necessary and sufficient conditions for rational and Du Bois singularities
on normal varieties and provide a characterization for these classes of
singularities in the Cohen--Macaulay case.

We recall that a variety $X$ has \emph{rational singularities} if given a
resolution of singularities $f \colon Y \to X$ such that $f_*\O_Y = \O_X$
and $R^if_*\O_Y = 0$ for $i > 0$, or in other words, such that the natural
map $\O_X \to Rf_*\O_Y$ is a quasi-isomorphism. The original definition of
Du Bois singularities is more complicated, and we will not recall it here.
Several alternative definitions were given by many authors throughout the
years, and we will adopt here the one given in \cite{Sch07} for which a
reduced scheme $X$ embedded in a smooth variety $M$ has \emph{Du Bois
singularities} if and only if, given a log resolution $g \colon N \to M$ of
$(M,\I_X)$ (note: not an embedded log resolution of $X$) that is an
isomorphism away from $X$, and denoting by $F = (g^{-1}(X))_\red$ the
reduced pre-image of $X$, the natural map $\O_X \to Rg_*\O_F$ is a
quasi-isomorphism. For more generalities on these classes of singularities,
we refer to \cite{KM98,Sch07,KSS10,KK10} and the references therein.

\subsection{Necessary condition and
characterization on Cohen--Macaulay varieties}

It is a well-known result of Kempf \cite{KKMSD73} that a normal variety $X$
has rational singularities if and only if it is Cohen--Macaulay and
$f_*\om_Y = \om_X$. An analogous property proven more recently by Kov\'acs,
Schwede and Smith says that a normal Cohen--Macaulay variety $X$ has Du
Bois singularities if and only if $f_*\om_Y(E) = \om_X$ where $f \colon Y
\to X$ is a log resolution and $E$ is the reduced exceptional divisor, see
Theorem~1.1 of \cite{KSS10}. Furthermore, it follows by Theorem 3.8 of
\cite{KSS10} that the identity $f_*\om_Y(E) = \om_X$ holds for all normal
varieties with Du Bois singularities, regardless of whether or not they are
Cohen--Macaulay.

These facts motivate the following result. 

\begin{thm}
\label{t:char}
Let $X$ be a normal variety, and let $\lcid X \subset \O_X$ be the
lci-defect ideal of $X$. Let $f \colon Y \to X$ be a log resolution of $X$,
and denote by $E$ the reduced exceptional divisor. Then the following
properties hold: 
\begin{enumerate}
\item
The pair $(X,\lcid X^{-1})$ is J-canonical if and only if $f_*\om_Y = \om_X$. 
\item
The pair $(X,\lcid X^{-1})$ is log J-canonical if and only if $f_*\om_Y(E) = \om_X$. 
\end{enumerate}
\end{thm}

\begin{proof}
Recall that $f_*\om_Y = \grom X$. By Theorem~\ref{t:can}, $\grom X = \om_X$
if and only if $\MIj{\lcid X^{-1}} = \O_X$, which is equivalent to
$(X,\lcid X^{-1})$ being J-canonical. This proves~(a). To prove~(b), first
note that if $f$ is an isomorphism over the regular locus of $X$ and $\l$
is a sufficiently small positive number, then $\rd{Z(\Jac_X^{-\l}\.\O_Y)} =
-E$, and thus
\[
f_*\om_Y(E) = \grom {(X,\Jac_X^{-\l})} \for 0 < \l \ll 1.
\]
Therefore $(X,\lcid X^{-1})$ is log J-canonical if and only if $\grom
{(X,\Jac_X^{-\l})} = \om_X$ for every sufficiently small $\l > 0$. On the
other hand, by definition of multiplier ideal and the fact that the radical
of $\lcid X$ contains $\Jac_X$, we have that $(X,\lcid X^{-1})$ is log
J-canonical if and only if $\MIj{\Jac_X^{-\l}\.\lcid X^{-1}} = \O_X$ for
every $\l > 0$. Therefore~(b) follows from Theorem~\ref{t:can(X,Z)}.
\end{proof}

\begin{cor}
\label{c:rat-DB}
Let $X$ be a normal variety, and let $\lcid X \subset
\O_X$ be the lci-defect ideal of $X$. 
\begin{enumerate}
\item
If $X$ has rational singularities then $(X,\lcid X^{-1})$ is
J-canonical.
\item
If $X$ has Du Bois singularities then $(X,\lcid X^{-1})$ is log
J-canonical.
\end{enumerate}
Moreover, the converse holds in both cases whenever $X$ is Cohen--Macaulay.
\end{cor}

In the special case when $X$ is $\Q$-Gorenstein, we obtain the result 
stated in the Introduction as Theorem~A.

\begin{cor}
\label{c:rat-DB-QGor}
With the same assumptions of Corollary~\ref{c:rat-DB}, suppose that $rK_X$
is Cartier for some positive integer $r$, and let $\lcid {r,X}$ be the
lci-defect ideal of level $r$ of $X$. 
\begin{enumerate}
\item
If $X$ has rational singularities then $(X,\lcid{r,X}^{1/r}\.\lcid
X^{-1})$ is canonical.
\item
If $X$ has Du Bois singularities then $(X,\lcid{r,X}^{1/r}\.\lcid
X^{-1})$ is log canonical.
\end{enumerate}
Moreover, the converse holds in both cases whenever $X$ is Cohen--Macaulay.
\end{cor}

If $\om_X$ is invertible then we can take $r=1$, and since in this case
$\lcid {1,X} = \lcid X$ the corollary recovers the well-known
characterization of rational singularities and Du Bois singularities on
Gorenstein varieties.

In general, assuming a priori that the variety is Cohen--Macaulay,
Corollary~\ref{c:rat-DB-QGor} gives new proofs to the facts that a variety
with log terminal (resp., log canonical) singularities has rational (resp.,
Du Bois) singularities, which we know from the results of
\cite{Elk81,KSS10,KK10}. To see this, first notice that
$(X,\lcid{r,X}^{1/r}\.\lcid X^{-1})$ is canonical if and only if it is log
terminal, since for every prime divisor $E$ over $X$
\[
\au E{X,\lcid{r,X}^{1/r}\.\lcid X^{-1}} = \aj E{X,\lcid X^{-1}} \in \Z
\]
by Proposition~\ref{p:Q-Gor}. Since $\lcid X^r \subset \ov{\lcid{r,X}}$ by
Proposition~\ref{p:lcid-r}, we have $\au E{X,\lcid{r,X}^{1/r}\.\lcid
X^{-1}} \ge \au EX$, and thus the aforementioned properties follow, under
the Cohen--Macaulay hypothesis, from the corollary. 

More interestingly, the corollary provides the necessary correction on
discrepancies for the converses of such results to hold.


\begin{rmk}
\label{r:strict}
One should think of the difference between $\lcid X$ and
$\lcid{r,X}^{1/r}$, from a valuation theoretic point of view, as the cause
of failure of the converses in the theorems in \cite{Elk81,KSS10,KK10}. Any
$\Q$-Gorenstein variety with rational singularities that is not log
terminal (for instance, the cone over an Enriques surface embedded by a
sufficiently positive line bundle) gives an instance where the inclusion
$\ov{\lcid X^r} \subset \ov{\lcid{r,X}}$ is strict.
\end{rmk}

\subsection{On the Cohen--Macaulay condition}

We discuss here an example showing that the Cohen--Macaulay hypothesis
cannot be dropped in the closing assertions of the above corollaries. The
example is known to the experts. For the convenience of the reader, we
first review some facts about cone singularities. 

To fix notation, let $S$ be a smooth projective variety of dimension $n-1
\ge 2$, embedded in a projective space by a projectively normal ample line
bundle $\O_S(1)$. Let then $X = \Spec \bigoplus_{m \ge 0}
H^0\big(\O_S(m)\big)$ be the cone over $S$, and let $f\colon Y \to X$ be
the resolution given by the total space of $\O_S(-1)$. The zero section of
such line bundle is the exceptional divisor $E$ of $f$. Note that $X$ is
normal. Let $x \in X$ be the vertex of the cone. 

We have 
\[
\om_X \cong \bigoplus_{m \in \Z} H^0\big(\om_S(m)\big) \and
f_*\om_Y \cong \bigoplus_{m > 0} H^0\big(\om_S(m)\big)
\]
by Theorem~(2.8) of \cite{Wat81} and Proposition~(1.6) of \cite{Wat83}, and therefore
$f_*\om_Y = \om_X$ if and only if $H^0\big(\om_S(m)\big) = 0$ for $m \le 0$.
One can similarly see that 
$f_*\om_Y(E) = \om_X$ if and only if $H^0\big(\om_S(m)\big) = 0$ for $m < 0$, 
but we will not use this fact.

It is well-known that $X$ is Cohen--Macaulay if and only if
$H^i\big(\O_S(m)\big) = 0$ for $n > i > 0$ and $m \ge 0$, and the
singularity is rational if and only if the same vanishing holds for $i>0$
and $m \ge 0$. It was proven by Du Bois \cite{DB81} (see also \cite{Ste83})
that $x\in X$ is a Du Bois singularity if and only if the natural map
$R^if_*\O_Y \to R^if_*\O_E$ is an isomorphism for all $i > 0$, or
equivalently, if and only if $R^if_*\O_Y(-E) = 0$ for $i > 0$. We will use
the following consequence of this property, which we learned from Karl
Schwede. 

\begin{lem}
\label{l:DB-cone}
With the above notation, if $X$ has Du Bois singularities then 
$H^i\big(\O_S(m)\big) = 0$ for all $n > i > 0$ and $m > 0$. 
\end{lem}

\begin{proof}
By Lemma~(2.3) of \cite{Wat83}, 
\[
H^{i+1}_x\big(\O_X\big) \cong \bigoplus_{m \in \Z} H^i\big(\O_S(m)\big).
\]
Note that $H^{i+1}_x\big(\O_X\big) \cong H^{i+1}_x\big(\frm\big)$, where
$\frm = f_*\O_Y(-E)\subset\O_X$ is the maximal ideal of $x$. The vanishing
of the higher direct images of $\O_Y(-E)$ gives the degeneration of the
appropriate Leray spectral sequence, and thus using duality (see
Proposition~(11.6) of \cite{Kol97}) in combination with the relative
version of the Grauert--Riemenschneider vanishing theorem, we have
\[
H^{i+1}_x\big(\O_X\big) \cong H^{i+1}_E\big(\O_Y(-E)\big) \overset{\rm dual}\sim
R^{n-i-1}f_*\om_Y(E) \cong R^{n-i-1}f_*\om_Y 
\overset{\rm dual}\sim H^i_E\big(\O_Y\big) \cong H^i\big(\O_S\big).
\]
The assertion follows by comparing the two formulas. 
\end{proof}

We are now ready to discuss the example.

\begin{eg}
\label{eg:example}
Let $C$ be a non-hyperelliptic curve of genus $g$. Fix a non-special
divisor $B$ on $C$ such that $\deg B \ge 2g+1$ (note that linear system
$|2(B-K_C)|$ is very ample and thus contains smooth elements), and let $\cE
= \O_C \oplus \om_C(-B)$. Following \cite{FGP05}, the ruled surface $\p
\colon S = \P_C(\cE) \to C$ is a \emph{canonical geometrically ruled
surface}. Moreover, if $H = C_0 + \p^*B$ where $C_0$ is the section
determined by the quotient $\cE \to \O_C$, then the linear system $|H|$
determines a projectively normal embedding of $S$ as a \emph{canonical
scroll} in some $\P^N$ (see Theorem~6.16 of \cite{FGP05}). Let $X$ be the
cone over $S \subset \P^N$. Since $H^1\big(\O_S(1)\big) \cong
H^1\big(C,\cE\big) \ne 0$, the singularity is not Du Bois (and thus not
rational) by Lemma~\ref{l:DB-cone}. On the other hand we have
$H^0\big(\om_S(m)\big) = 0$ for all $m \le 0$, and so $f_*\om_Y = \om_X$.
We conclude that the pair $(X,\lcid X^{-1})$ is J-canonical (and thus log
J-canonical) by Theorem~\ref{t:char}.  
\end{eg}

\subsection{A general sufficient condition}

Using inversion of adjunction and the results of \cite{Kaw98,KK10} in place
of Theorem~\ref{t:can}, one obtains the following general sufficient
condition. The argument was brought to our attention by Mircea \Mustata.

\begin{thm}
\label{t:rat-DB-inv-adj}
Let $X$ be a reduced equidimensional scheme. 
\begin{enumerate}
\item
If $X$ is J-canonical, then $X$ is the disjoint union of its irreducible
components, each of which is log terminal in the sense of \cite{dFH09} and
has rational singularities. In particular, $X$ is normal and
Cohen--Macaulay.
\item
If $X$ is log J-canonical, then $X$ has Du Bois singularities.
In particular, $X$ is semi-normal.  
\end{enumerate}
\end{thm}

\begin{proof}
We can assume that $X$ is embedded in a smooth variety $M$, with
codimension $e$. If $X$ is J-canonical then by Theorem~\ref{t:inv-adj} (see
Remark~\ref{r:inv-adj} if $\dim X = 1$) the pair $(M,eX)$ is canonical, and
thus each irreducible component of $X$ is an isolated log canonical center
of the pair. Since intersections of log canonical centers are log canonical
centers, it follows in particular that $X$ is the disjoint union of its
irreducible components. Moreover, it follows by the main result of
\cite{Kaw98} that each irreducible component of $X$ has rational
singularities and is log terminal in the sense of \cite{dFH09}. If $X$ is
only log J-canonical, then $(M,eX)$ is log canonical (again by
Theorem~\ref{t:inv-adj} and Remark~\ref{r:inv-adj}) and $X$ is a log
canonical center of $(M,eX)$, so the result follows in this case by
Theorem~1.4 of \cite{KK10}. Regarding the last assertion in~(b), see
Remark~1.11 of \cite{KK10}.
\end{proof}

The above theorem is well-known to the specialists once the assumptions on
the singularities of $X$ are expressed in terms of the jet schemes of
$X_m$. Indeed, by Corollary~\ref{c:jet-sing}, the result can be rephrased
by saying that if $\dim X_m = n(m+1)$ for every $m$, then $X$ has Du Bois
singularities, and if moreover every irreducible component of maximal
dimension $n(m+1)$ of $X_m$ dominates an irreducible component of $X$, then
$X$ has rational singularities.


\vspace{0.2cm}
\setlength{\parindent}{0in}
\def\scshape{}

\end{document}